\documentclass{amsart}
\usepackage{latexsym}
\usepackage{rotating}
\usepackage{amsfonts}
\usepackage{amssymb}
\usepackage{amsmath}
\usepackage{graphics, color}
\usepackage{amsthm}
\usepackage{enumerate}
\input xy
\xyoption{all}
\DeclareMathOperator*{\colim}{colim}

\DeclareMathOperator*{\coprodmo}{\coprod}
\usepackage{url}

\newtheorem{theorem}{Theorem}[section]
\newtheorem{lemma}[theorem]{Lemma}
\newtheorem{corollary}[theorem]{Corollary}
\newtheorem{proposition}[theorem]{Proposition}
\newtheorem{question}[theorem]{Question}

\newtheorem{conjecture}[theorem]{Conjecture}

\newtheorem{definition}[theorem]{Definition}
\newtheorem{remark}[theorem]{Remark}

\begin{document}

\newcommand{\coker}{\mathrm{coker}}

\newcommand{\bbA}{\mathbb{A}}
\newcommand{\bbC}{\mathbb{C}}
\newcommand{\C}{\mathbb{C}}
\newcommand{\N}{\mathbb{N}}
\newcommand{\bbN}{\mathbb{N}}
\newcommand{\bbR}{\mathbb{R}}
\newcommand{\bbRP}{\mathbb{RP}}
\newcommand{\bbZ}{\mathbb{Z}}
\newcommand{\mcZ}{\mathcal{Z}}
\newcommand{\bbF}{\mathbb{F}}
\newcommand{\bbQ}{\mathbb{Q}}

\newcommand{\A}{\mathcal{A}}
\newcommand{\B}{\mathcal{B}}
\newcommand{\mcC}{\mathcal{C}}
\newcommand{\mcD}{\mathcal{D}}
\newcommand{\E}{\mathcal{E}}
\newcommand{\mF}{\mathcal{F}}
\newcommand{\G}{\mathcal{G}}
\newcommand{\mcG}{\mathcal{G}}
\newcommand{\mH}{\mathcal{H}}
\newcommand{\mcH}{\mathcal{H}} 
\newcommand{\I}{\mathcal{I}}
\newcommand{\mL}{\mathcal{L}}
\newcommand{\mcN}{\mathcal{N}}
\newcommand{\M}{\mathcal{M}}
\newcommand{\mO}{\mathcal{O}}
\newcommand{\mcP}{\mathcal{P}}
\newcommand{\mcR}{\mathcal{R}}
\newcommand{\T}{\mathcal{T}}
\newcommand{\U}{\mathcal{U}}
\newcommand{\V}{\mathcal{V}}
\newcommand{\Z}{\mathcal{Z}}

\newcommand{\bS}{\mathbf{S}}
\newcommand{\bd}{\mathbf{d}}

\newcommand{\bG}{\mathcal{G}_0} 
\newcommand{\sth}[1]{#1^{\mathrm{th}}}
\newcommand{\abs}[1]{\left| #1\right|}
\newcommand{\ord}[1]{\Delta \left( #1 \right)}
\newcommand{\leqs}{\leqslant}
\newcommand{\geqs}{\geqslant}
\newcommand{\heq}{\simeq}
\newcommand{\iso}{\simeq}
\newcommand{\maps}{\longrightarrow}
\newcommand{\lmaps}{\longleftarrow}
\newcommand{\injects}{\hookrightarrow}
\newcommand{\homeo}{\cong}
\newcommand{\surjects}{\twoheadrightarrow}
\newcommand{\longsurjects}{\twoheadlongrightarrow}
\newcommand{\isom}{\cong}
\newcommand{\cross}{\times}
\newcommand{\normal}{\vartriangleleft}
\newcommand{\wt}[1]{\widetilde{#1}} 
\newcommand{\fc}{\mathcal{A}_{\mathrm{flat}}} 
\newcommand{\flc}{\mathcal{A}_{\mathrm{fl}}}

\newcommand{\Rdef}{R^{\mathrm{def}}}
\newcommand{\Sym}{\mathrm{Sym}}
\newcommand{\vect}[1]{\stackrel{\rightharpoonup}{\mathbf #1}}
\newcommand{\SR}{\mathcal{SR}}
\newcommand{\SRe}{\mathcal{SR}^{\mathrm{even}}}
\newcommand{\Rep}{\mathrm{Rep}}
\newcommand{\SRep}{\mathrm{SRep}}
\newcommand{\Hom}{\mathrm{Hom}}
\newcommand{\HHom}{\mathcal{H}\mathrm{om}}
\newcommand{\Lie}{\mathrm{Lie}}
\newcommand{\K}{K^{\mathrm{def}}}
\newcommand{\mK}{\mathcal{K}_{\mathrm{def}}}
\newcommand{\SK}{SK_{\mathrm{def}}}
\newcommand{\dom}{\mathrm{dom}}
\newcommand{\codom}{\mathrm{codom}}
\newcommand{\Ob}{\mathrm{Ob}}
\newcommand{\Mor}{\mathrm{Mor}}
\newcommand{\+}[1]{\underline{#1}_+}
\newcommand{\Fin}{\Gamma^{\mathrm{op}}}
\newcommand{\f}[1]{\underline{#1}}
\newcommand{\hofib}{\mathrm{hofib}}
\newcommand{\Stab}{\mathrm{Stab}}
\newcommand{\wtStab}{\wt{\mathrm{Stab}}}
\newcommand{\Css}{\mathcal{C}_{ss}}
\newcommand{\Map}{\mathrm{Map}}
\newcommand{\bMap}{\mathrm{Map_*}}
\newcommand{\bdMap}{\mathrm{Map_*^\delta}}
\newcommand{\flatc}{\mathcal{A}_{\mathrm{flat}}}
\newcommand{\F}[1]{\mathrm{Flag}(\vect{#1})}
\newcommand{\p}{\vect{p}}
\newcommand{\avg}{\mathrm{avg}}
\newcommand{\smsh}[1]{\ensuremath{\mathop{\wedge}_{#1}}}
\newcommand{\ol}[1]{\overline{#1}}
\newcommand{\Vect}{\mathrm{Vect}}
\newcommand{\bv}{\bigvee}
\newcommand{\Gr}{\mathrm{Gr}}
\newcommand{\Mf}{\mathcal{M}_{\textrm{flat}}}
\newcommand{\ku}{\mathbf{ku}}
\newcommand{\Susp}{\Sigma}
\newcommand{\Id}{\textrm{Id}}
\newcommand{\id}{\textrm{Id}}
\newcommand{\xmaps}{\xrightarrow}
\newcommand{\srm}[1]{\stackrel{#1}{\maps}}
\newcommand{\srt}[1]{\stackrel{#1}{\to}}
\newcommand{\sm}{\wedge}
\newcommand{\conv}{\Rightarrow}
\newcommand{\Tor}{\textrm{Tor}}
\newcommand{\goesto}{\mapsto}
\newcommand{\nd}{\noindent}
\newcommand{\Ind}{\mathrm{Ind}}
\newcommand{\bInd}{\overline{\mathrm{Ind}}}
\newcommand{\R}{\mathrm{R}}
\newcommand{\bR}{\overline{\mathrm{R}}}
\newcommand{\bRf}{\overline{\mathrm{R}}^{\mathrm{free}}}
\newcommand{\ra}{\rangle}
\newcommand{\la}{\langle}
\newcommand{\Sum}{\mathrm{Sum}}
\newcommand{\Res}{\mathrm{Res}}
\newcommand{\Proj}{\mathrm{Proj}}
\newcommand{\GL}{\mathrm{GL}}
\newcommand{\PU}{\mathrm{PU}}
\newcommand{\Irr}{\mathrm{Irr}}
\newcommand{\rHom}{\Hom_A (\Gamma, U(n))}
\newcommand{\qcd}{\bbQ\mathrm{cd}}
\newcommand{\rk}{\mathrm{rank}}
\newcommand{\Irrf}{\Irr_n (H)^{\mathrm{free}}}
\newcommand{\brho}{\overline{\rho}}
\newcommand{\Sp}{\mathrm{Sp}}
\newcommand{\Span}{\mathrm{Span}}
\newcommand{\Img}{\mathrm{Im}}
\newcommand{\bSum}{\overline{\mathrm{Sum}}}
\newcommand{\bHom}{\overline{\mathrm{Hom}}}
\newcommand{\bIrr}{\overline{\mathrm{Irr}}}
\newcommand{\bIrrp}{\overline{\mathrm{Irr}}^+}
\newcommand{\Isom}{\mathrm{Isom}}
\newcommand{\tIrr}{\wt{\Irr}}
\newcommand{\pIrr}{\partial \tIrr}
\newcommand{\defn}{\mathrel{\mathop :}=}
\newcommand{\psubset}{\subsetneqq}
\newcommand{\propersubset}{\subsetneqq}

\def\co{\colon\thinspace}

\title[Representation theory of crystallographic groups]{Periodicity in the stable representation theory of crystallographic groups}
\author[D\,A Ramras]{Daniel A. Ramras\\
New Mexico State University\\
Las Cruces, New Mexico 88003-8001\\
ramras@nmsu.edu}

\address{New Mexico State University\\
Department of Mathematical Sciences\\
P.O. Box 30001\\
Department 3MB\\
Las Cruces, New Mexico 88003-8001 }
\email{ramras@nmsu.edu}
\urladdr{http://www.math.nmsu.edu/~ramras/}

\thanks{Partially supported by NSF grants DMS-0353640 (RTG),  DMS-0804553, and DMS-0968766.}




 \begin{abstract}
Deformation K--theory associates to each discrete group $G$ a spectrum built from spaces of finite dimensional unitary representations of $G$.  In all known examples, this spectrum is 2--periodic above the rational cohomological dimension of $G$ (minus 2), in the sense that T. Lawson's Bott map is an isomorphism on homotopy in these dimensions.  We establish a periodicity theorem for crystallographic subgroups of the isometries of $k$--dimensional Euclidean space.  For a certain subclass of torsion-free crystallographic groups, we prove a vanishing result for the homotopy groups of the stable moduli space of representations, and we provide examples relating these homotopy groups to the cohomology of $G$.

These results are established as corollaries of the fact that for each $n > 0$, the one-point compactification of the moduli space of irreducible $n$--dimensional representations of $G$ is a CW complex of dimension at most $k$.  This is proven using  real algebraic geometry and projective representation theory.

\end{abstract}

\maketitle{}


\section{Introduction}

Given a discrete group $\Gamma$, Carlsson's deformation $K$--theory spectrum $\K(\Gamma)$ provides a stable--homotopical setting in which to study the unitary representation spaces $\Hom(\Gamma, U(n))$.  
This spectrum can be described as the $K$--theory of the topological permutative category of $U(n)$--representations of $\Gamma$, and its zeroth space admits an explicit description in terms of representation spaces (see Ramras \cite[Section 2]{Ramras-excision}).
For products of aspherical surface groups, $\K (\Gamma)$ is 2-periodic above the rational cohomological dimension of $\Gamma$ minus 2 (Ramras \cite{Ramras-stable-moduli}).  Specifically, T. Lawson's Bott map \cite{Lawson-prod-form} induces an isomorphism in homotopy
$$\beta_* \co \pi_* \K (\Gamma)  \srm{\isom}  \pi_{*+2} \K(\Gamma)$$
for $*> \qcd (\Gamma) - 2$.
One consequence of this result is that the stable moduli space
$$\Hom(\Gamma, U)/U \homeo \colim_n \Hom (\Gamma, U(n))/U(n)$$
has vanishing homotopy in dimensions greater than $\qcd(\Gamma)$.  For an aspherical, closed manifold $M$ with fundamental group $\Gamma$, this space is also the stable moduli space $\Mf(M)$ of flat, unitary connections on bundles over $M$.  
In this article we prove periodicity and vanishing results for new classes of groups.  Our periodicity result applies to any group with a finite index subgroup isomorphic to $\bbZ^k$ for some $k\geqs 0$; note that two such finite index subgroups always have the same rank.  We say that such groups are \emph{virtually} $\bbZ^k$.  All crystallographic groups $\Gamma < \Isom (\bbR^k)$ are virtually $\bbZ^k$.

\begin{theorem}[Section \ref{period}]$\label{finite-ext}$
If $\Gamma$ is virtually $\bbZ^k$ for some $k\geqs 0$, then the Bott map 
$$\beta_* \co \pi_*\K (\Gamma) \srm{\isom} \pi_{*+2} \K (\Gamma)$$
is an  isomorphism  for $*> k-2$.
\end{theorem}

For $\Gamma$ a product of aspherical surface groups, it was shown using Yang--Mills theory that  in dimensions $*>\qcd(\Gamma) - 2$ there are isomorphisms $\pi_* \K(\Gamma)\isom K^{-*} (B\Gamma)$, where $K^{-*} (B\Gamma)$ is the complex $K$--theory of $B\Gamma$ (Ramras~\cite{Ramras-surface, Ramras-stable-moduli}).   By analogy, we make the following conjecture.

\begin{conjecture}$\label{conj}$
If $\Gamma < \Isom (\bbR^k)$ is a torsion-free crystallographic group, then for $*>\qcd (\Gamma)-2$ there is an isomorphism $\pi_* \K (\Gamma) \isom K^{*} (\bbR^k/\Gamma)$.
\end{conjecture}

If $\Gamma$ is virtually $\bbZ^k$, a  transfer argument shows that $H^*(\Gamma; \bbQ) = 0$ for $*>k$.  Any finite index subgroup of $\Gamma$ is also virtually $\bbZ^k$, so the virtual rational cohomological dimension of $\Gamma$ is precisely $k$ (but the rational cohomological dimension of $\Gamma$ could be less than $k$).  When $\Gamma$ is crystallographic and torsion-free, $\bbR^{k} \to \bbR^{k}/\Gamma$ is a covering space (Wolf \cite[Theorem 3.1.3]{Wolf}).  The Euclidean space form $\bbR^{k}/\Gamma$ is a model for the classifying space $B\Gamma$, so if 
$\bbR^{k}/\Gamma$ is orientable, then the rational and integral cohomological dimensions of $\Gamma$ are both $k$.   It would be interesting if Theorem \ref{finite-ext} could be improved in the case when $\bbR^k/\Gamma$ is non-orientable, since then $\qcd(\Gamma)<k$; we obtain such results for one infinite family in Section \ref{examples}.
Using methods rather different from those used here, it will be shown in future work of the author that the isomorphism in Conjecture~\ref{conj} holds \emph{rationally} in dimensions $*>k-2$.

It is important to note that Conjecture~\ref{conj} \emph{fails} in general.  In \cite[\S4.2]{Lawson-simul}, Lawson shows that for the integral Heisenberg group, the Bott map is an isomorphism in dimensions greater than zero, but the periodic groups are much larger than the $K$--theory of the classifying space.  Examples of groups with trivial deformation $K$--theory and non-trivial complex $K$--theory appear in Ramras~\cite[\S1]{Ramras-stable-moduli}.


Theorem \ref{finite-ext}, Conjecture \ref{conj}, and the results on surface groups discussed above are analogous to the the Atiyah--Segal Theorem and to the Quillen--Lichtenbaum Conjecture~\cite{Levine, OR} in algebraic $K$--theory (now a theorem of Voevodsky), which states that the algebraic $K$--theory of a scheme should agree with \'{e}tale $K$--theory (mod $l$) in dimensions greater than the (virtual, mod $l$) \'{e}tale cohomological dimension minus 2.    
 Conjectures of Carlsson~\cite{Carlsson-derived-rep} (see also the introduction to Lawson~\cite{Lawson-prod-form}) link deformation $K$--theory to algebraic $K$--theory of fields, while \'{e}tale $K$--theory bears many similarities to ordinary complex $K$--theory (for example, Soul\'{e}'s \'{e}tale Chern character~\cite{Soule}). Recent work of the author and T. Baird (in preparation) explains the appearance of \emph{rational} cohomology in this picture, by showing that classes in $H^*(\Gamma; \bbQ)$ provide obstructions to surjectivity of a natural map 
$\alpha_*\co \pi_* \K (\Gamma) \to K^* (B\Gamma)$.   

Lawson's work establishes a close relationship between $\K(\Gamma)$ and the space 
$$\Hom(\Gamma, U)/U \homeo \colim_n \Hom (\Gamma, U(n))/U(n),$$
which we call the \emph{stable moduli space} of representations.
In good cases (see Lemma \ref{stably-gplike} and Theorem \ref{Bott}), there is an Atiyah--Hirzebruch style spectral sequence converging from $\pi_* (\bbZ\cross \Hom(\Gamma, U)/U)\otimes \pi_* \ku$ to $\pi_* \K(\Gamma)$.  (In general, the $E_2$ term has a more complicated description; see Theorem~\ref{Bott}.)  This spectral sequence is analogous to the truncated Beilinson--Bloch--Lichtenbaum spectral sequence \cite{Bloch-Lichtenbaum} (see also Suslin~\cite{Suslin-SS}, Levine~\cite{Levine-SS}), which converges from motivic cohomology to algebraic K--theory.  Motivic cohomology is often viewed as an integral cohomology theory for schemes, so one might expect an analogous relationship between $\pi_* \Hom(\Gamma, U)/U$ and $H^*(\Gamma; \bbZ)$.  These groups agree up to torsion when $\Gamma$ is a product of aspherical surface groups (Ramras~\cite{Ramras-stable-moduli}), and a similar result is established for the examples in Section \ref{examples}.  Viewed as a comparison between the $E_2$ terms of the Atiyah--Hirzebruch spectral sequences for deformation $K$--theory and complex $K$--theory, these results are then in analogy with comparisons between motivic and \'etale cohomology (for example, Mazza--Voevodsky--Weibel \cite[Theorem 10.2]{MVW}).
As a further step in this direction, we  use work of  Ratcliffe and Tschantz \cite{R-T} to show that if $E$ is a flat torus bundle over a torus, then
$\pi_* \Hom(\pi_1 E, U)/U$
vanishes (integrally) above the dimension of $E$ (Theorem \ref{stable-moduli}).  This applies to more general manifolds formed by replacing the base torus with a flat torus bundle over a torus (and so on).

The proofs of these results depend on Lawson's spectral sequences \cite{Lawson-prod-form, Lawson-simul}, which allow us to deduce our periodicity and vanishing results from the vanishing of $H_* (\bIrrp_n (\Gamma); \bbZ)$ for $*>k$, where $\bIrrp_n (\Gamma)$ is the one-point compactification of the moduli space of irreducible $U(n)$--representations.
The main result of this article, then, is the following dimension bound.

\begin{theorem}[Section \ref{proofs}]$\label{finite-cplx}$  If $\Gamma$ is virtually $\bbZ^k$, then for every $n>0$, the space $\bIrrp_n (\Gamma)$ has a CW structure of dimension at most $k$.
\end{theorem}

This bound is optimal, in the sense that if $\Gamma$ has a subgroup $A\isom \bbZ^k$ of index $q$,  
then $\bIrrp_q (\Gamma)$ is a CW complex of dimension precisely $k$ (Remark~\ref{max-dim'l}).
Low-dimensionality seems particular to the space of \emph{irreducible} representations: since commuting matrices are simultaneously diagonalizable, there is a homeomorphism $\Hom(\bbZ^k, U(n))/U(n) \homeo \Sym^n((S^1)^k)$, obtained by recording eigenvalues (here $\Sym^n$ is the $n^\textrm{th}$ symmetric product).  Generically, this space is a manifold of dimension $nk$.  Note that in this case, there are no irreducible representations (unless $n=1$, when we simply have a $k$--dimensional torus).

The necessary CW structures are produced using real algebraic geometry.  The bound on dimensions comes from combining information about the induction maps 
with projective representation theory of finite groups.   

\vspace{.1in}
\nd {\bf Organization:}  In Section \ref{real} we review facts from real algebraic geometry and show (following Schwarz \cite{Schwarz-smooth}) that $\Hom (\Gamma, U(n))/U(n)$ is semi-algebraic.  Section \ref{crystal} reviews crystallographic groups and Section \ref{rep-spaces} studies spaces of representations.  In Section \ref{ind-proj} we split the irreducible representations of $\Gamma$ into two classes: induced representations and those yielding projective representations of a finite quotient $\Gamma/A$.  We study these two classes in Sections \ref{ind} and \ref{proj}.  In Section \ref{proofs}, we prove the main result on the dimension of $\bIrrp_n (\Gamma)$ (Theorem \ref{finite-cplx}), and in Section \ref{period} we prove our periodicity theorem (Theorem \ref{finite-ext}).  The application to stable moduli spaces appears in Section \ref{sm-sec}, and Section~\ref{examples} contains some explicit computations.

\vspace{.1in}

\nd {\bf Acknowledgements:}  I thank Fred Cohen for suggesting crystallographic groups as a source of examples in deformation $K$--theory, Tyler Lawson and Qayum Khan for several helpful conversations regarding representation theory, and Sean Lawton for pointing out Schwarz's work on quotient spaces.  Comments from the anonymous referees helped to improve the exposition.


\section{Background on real algebraic geometry} $\label{real}$

In this section we review  concepts and results from real algebraic geometry needed in the sequel, following Bochnak, Coste, and Roy~\cite{BCR}  (especially Chapters 2 and 9).  In the last subsection, we discuss a result of Schwarz on quotients of semi-algebraic sets by linear actions of compact Lie groups.

By definition, a semi-algebraic subset of $\bbR^N$ is a finite union of sets of the form
$$\{(x_1, \ldots, x_N) \in \bbR^N \, |\, p_1 (x_1, \ldots, x_N) \sim_1 0, \ldots, p_m (x_1, \ldots, x_N) \sim_n 0\},$$
where the $p_i$ are polynomials and the relations $\sim_i$ are either $>$, $<$, or $=$.  The relations $\geqs$ and $\leqs$ are unnecessary since we allow finite unions, and when only the relation $=$ is used, we obtain a real algebraic variety.
For every finitely generated group $\Gamma$, the representation space $\Hom(\Gamma, U(n))$ is a real algebraic variety, cut out by the equations defining $U(n)$ and the relations in $\Gamma$.  (Note that by the Hilbert Basis Theorem, every ideal in the polynomial ring $\bbR[x_1, \ldots, x_n]$ is finitely generated, so there is no need to assume that $\Gamma$ is finitely presented.)

The appropriate notion of morphisms between semi-algebraic subsets $X\subset \bbR^N$ and $Y\subset \bbR^M$ is that of \emph{semi-algebraic functions}.  By definition, a function $f\co X\to Y$ is semi-algebraic if its graph 
$$\Gr (f) = \{(x,y) \in \bbR^N \cross \bbR^M \, |\, x\in X \textrm{ and } y = f(x) \}$$
is a semi-algebraic subset of $\bbR^N \cross \bbR^M$.  Note that, for example, any polynomial mapping $\bbR^N \to \bbR^M$ is semi-algebraic when restricted to a semi-algebraic subset of $\bbR^N$.  We will see below that composites of semi-algebraic maps are semi-algebraic.

We now list the basic closure properties for semi-algebraic sets that we will need.  These results can all be found in \cite[Chapter 2]{BCR}.

\begin{theorem} $\label{closure}$ Let $X\subset \bbR^N$ be semi-algebraic.

\begin{itemize}
\item Complementation: If $A\subset X$  is semi-algebraic, then so is $X \setminus A$.

\item Products: If $Y \subset \bbR^M$ is semi-algebraic, then so is $X\cross Y \subset \bbR^{N+M}$.

\item Interior: The interior of $X$ is semi-algebraic.

\item Tarski--Seidenberg Theorem: Say $f\co X \to \bbR^M$ is semi-algebraic.  Then its image $f(X)$ is a semi-algebraic subset of $\bbR^M$, and if $Y \subset \bbR^M$ is semi-algebraic, then $f^{-1} (Y)$ is semi-algebraic as well.
\end{itemize}
\end{theorem}

The Tarski--Seidenberg Theorem follows from the special case of coordinate projections $\bbR^{l} \to \bbR^k$ ($k<l$): given $f\co X\to \bbR^M$ and $Y\subset \bbR^M$, $f(X)$ is the projection of $\Gr(f)$ onto $\bbR^M$ and $f^{-1} (Y)$ is the projection of $\Gr(f)\cap (X \cross Y)$ onto $\bbR^N$.

Using Theorem \ref{closure}, we can now check that composites of semi-algebraic maps are semi-algebraic.  If $f\co X\to Y$ and $g\co Y\to W$ are semi-algebraic, then the graph $\Gr(f\circ g)$ is simply the image of $\Gr(f)$ under the map $\Id_X \cross g$.  Moreover, the graph of a product is simply product of the graphs (up to reordering the coordinates) so products of semi-algebraic maps are semi-algebraic.

The Tarski--Seidenberg Theorem is intimately related to \emph{elimination of quantifiers}.  One special case will be useful to us: if $Z\subset \bbR^{N+M}$ is semi-algebraic, then 
$$X = \{x\in \bbR^N \, |\, \exists y\in \bbR^M \textrm{ such that } (x,y) \in Z\}$$
is semi-algebraic in $\bbR^N$:   $X$ is  the projection of $Z\subset \bbR^{N+M}$ onto $ \bbR^{N}$.  
Another application is that affine simplices in $\bbR^N$ are semi-algebraic.  
Recall that if $v_0, \ldots, v_d\in \bbR^N$ are affinely independent (meaning no $v_i$ is a linear combination of the other $v_j$ with coefficients summing to 1) then  
$$\langle  v_0, \ldots, v_d\rangle := \{x \in \bbR^N \,|\, \exists t_1,\ldots, t_d\in \bbR \textrm{ with } t_i \geqs 0, \, \sum t_i = 1, \, \sum t_i v_i = x\}.$$
Eliminating the existential quantifier shows that $\langle  v_0, \ldots, v_d\rangle$ is semi-algebraic.  

\subsection{Triangulations} $\label{triang-sec}$

Triangulations of semi-algebraic sets are discussed in \cite[Chapter 9]{BCR}, and also in Hironaka's article \cite{Hironaka}.  We need some terminology regarding simplicial complexes.

A finite \emph{affine simplicial complex} $K \subset \bbR^M$ consists of a finite union of affine simplices, 
in which each pair of simplices intersect either trivially or in a common face.  Since affine simplices are semi-algebraic, so are finite affine simplicial complexes.

We refer to the interior of an affine simplex in $\bbR^N$ as an \emph{open affine simplex}.  By convention, the interior of a zero-dimensional simplex $\langle v \rangle$ is simply $\langle v \rangle$.  Note that open affine simplices are semi-algebraic, as can be seen from the equations defining them, or from the fact that interiors of semi-algebraic sets are always semi-algebraic.  Hence any union of open simplices inside a simplicial complex is semi-algebraic.

\begin{definition}$\label{open}$
We define a finite \emph{affine open simplicial complex} to be a union of open affine simplices inside some finite affine simplicial complex.  
The dimension of such a complex $W$ is the maximum dimension of a simplex in $W$.

A finite \emph{open triangulation} of a space $X$ is a homeomorphism $f\co W\srt{\homeo} X$ with $W$ a finite affine open simplicial complex.  We call such a triangulation \emph{semi-algebraic} if $f$ is a semi-algebraic map.  The images, under $f$, of open simplices in $W$ will be called open simplices of $X$.
\end{definition}

The open simplices of a finite affine open simplicial complex $W \subset \bbR^N$ are partially ordered by containment of their closures.  If $W\srt{\homeo} X$ is an open triangulation,   there is an induced partial order on the open simplices of $X$.  The \emph{maximal} open simplices are then  open subsets of $X$.
The following result, proven in~\cite[Theorem 9.2.1]{BCR} and in~\cite{Hironaka}, shows that compact semi-algebraic sets can be triangulated in a manner compatible with any finite family of semi-algebraic subsets.

\begin{theorem} $\label{triang-thm}$ If $X \subset \bbR^N$ is a compact semi-algebraic subset, and $Y_1, \ldots, Y_k \subset X$ are semi-algebraic subsets of $X$, then there exists a semi-algebraic triangulation of $X$ in which each $Y_i$ is a union of open simplices.
\end{theorem}

Semi-algebraic mappings are tame, in the sense that they do not increase dimension (loosely speaking, this means that space filling maps cannot be semi-algebraic).  To make this precise, we need to review the notion of dimension in real algebraic geometry.  This material can be found in \cite[Chapter 2]{BCR}.

Recall that the (Krull) dimension of a ring $R$ is the length of the longest chain of prime ideals in $R$, where a chain $P_0 \psubset P_1 \ldots \psubset P_n$ has length $n$.

\begin{definition}
The (algebraic) \emph{dimension} of a semi-algebraic subset $X\subset \bbR^N$ is the dimension of the coordinate ring $\bbR[t_1, \ldots, t_N]/\mathcal{I} (X)$, where $\I (X)$ denotes the ideal of polynomials vanishing on $X$.
\end{definition}  

Here is the relevant result about algebraic dimension (see \cite[Theorem 2.8.8]{BCR}).

\begin{theorem} $\label{alg-tame}$ If $f\co X\to Y$ is a surjective semi-algebraic mapping between semi-algebraic sets, then $\dim(Y) \leqs \dim(X)$.
\end{theorem}

As discussed above, open affine simplices are semi-algebraic, and an open affine simplex with $d+1$ vertices has (algebraic) dimension precisely $d$.  This follows from the more general fact \cite[Theorem 2.8.14]{BCR} that semi-algebraic sets in $\bbR^M$ which are smooth $d$--dimensional submanifolds always have (algebraic) dimension $d$.  

\begin{lemma} $\label{dimn}$ If $f\co W \srt{\homeo} X\subset \bbR^N$ is a semi-algebraic open triangulation of a semi-algebraic set, then $\dim (X)$ equals the maximum dimension of a simplex in $W$.
\end{lemma}
\begin{proof}If $A = \bigcup_{i=1}^p A_i$ with each $A_i$  semi-algebraic, then $\dim (A) = \max_i \dim(A_i)$ (\cite[Theorem 2.8.5]{BCR}).  Each open simplex $\sigma \subset W$ is semi-algebraic, and the triangulation $f$ is semi-algebraic, so each $f(\sigma)$ is semi-algebraic  and the result follows.
\end{proof}

Using Lemma \ref{dimn}, we can now restate Theorem \ref{alg-tame} in terms of triangulations.

\begin{corollary}$\label{tame}$ If $f\co X \to Y$ is a surjective semi-algebraic map between semi-algebraic sets, and $X$ admits a semi-algebraic open triangulation of dimension $k$, then every semi-algebraic open triangulation of $Y$ has dimension at most $k$.
\end{corollary}

This result will be applied to induction maps  in Section \ref{proofs}.

\subsection{Quotients of semi-algebraic sets by linear actions of compact Lie groups}  In this section, we discuss a result of Schwarz \cite{Schwarz-smooth} (see also Procesi--Schwarz \cite{PS-conf}) regarding quotients of semi-algebraic sets.
Let $K$ be a compact Lie group, and let $\rho\co K \to \GL_N (\bbR)$ be a representation.   

\begin{theorem}[Schwarz] $\label{Schwarz}$  If $X\subset \bbR^N$ is a $K$--invariant semi-algebraic subset, then the quotient $X/K$ is homeomorphic to a semi-algebraic set, and under this homeomorphism, the map $X\to X/K$ becomes semi-algebraic.
\end{theorem}

\begin{corollary} $\label{bHom-sa}$
For any finitely generated group $\Gamma$, the moduli space of representations 
$\bHom_n (\Gamma) = \Hom(\Gamma, U(n))/U(n)$
has the structure of a semi-algebraic set, with a semi-algebraic quotient map
$$\Hom (\Gamma, U(n)) \maps \bHom_n (\Gamma).$$
\end{corollary}
\begin{proof} If $\Gamma$ is generated by $k$ elements, then  
$\Hom(\Gamma, U(n)) \subset U(n)^k \subset (M_{2n} (\bbR))^k,$
is a compact algebraic subset, and the conjugation action of $U(n)$ extends to a linear action on $(M_{2n} (\bbR))^k$.  Hence Theorem \ref{Schwarz} applies.
\end{proof}

We  sketch the proof of Theorem~\ref{Schwarz}.  Most of the ideas are found in Schwarz~\cite{Schwarz-smooth}, or in Weyl~\cite[Chapter 8, \S 14]{Weyl}.  One key ingredient is the following  well-known lemma (see, for example, \cite{Weyl} or Sepanski \cite[Exercise 7.35, p. 186]{Sepanski}).

\begin{lemma} $\label{avg}$ Let $p\in \bbR[x_1, \ldots, x_N]$ be a polynomial.  Then the function
$$\bar{p} (x_1, \ldots, x_N) = \int_K p(k\cdot (x_1, \ldots, x_N)) dk$$
is a $K$--invariant polynomial function.  Here $dk$ denotes the (right-invariant) Haar measure on $K$, with the total measure of $K$ normalized to one.
\end{lemma}

\begin{proof} Invariance of $\bar{p}$ is immediate from invariance of Haar measure.  One sees that $\bar{p}$ is a polynomial by writing out the action of $k\in K$ as multiplication by the matrix $\rho(k) = [k_{ij}] \in GL_N (\bbR)$.  The coefficients of the polynomial $\bar{p}$ are then integrals, over $K$, of various polynomials in the entries $k_{ij}$.
\end{proof}

Continuing, let $\mcP = \bbR[x_1, \ldots, x_n]$, and let $\mcP^K$ denote the subalgebra of $\mcP$ consisting of $K$--invariant polynomials.  First, we check that $\mcP^K$ is finitely generated as an algebra over $\bbR$.  By the Hilbert Basis Theorem, the ideal $I$ generated by the non-constant elements of $\mcP^K$ is generated by a finite number of polynomials of the form $J_i = \sum p_{ij} q_{ij}$ with $q_{ij}\in \mcP^K$ and $q_{ij}$ non-constant.  The collection $\{q_{ij}\}_{i,j}\subset \mcP^K$ is now a finite, invariant generating set for the ideal $I$.  

We claim that $\{q_{ij}\}_{i,j}$ generates $\mcP^K$ as an algebra.  If $q\in \mcP^K$ is a non-constant invariant, then we can write $q = \sum_{i,j} a_{ij} q_{ij}$ for some polynomials $a_{ij}$, and now
$$q = \bar{q} = \sum_{i,j} \overline{a_{ij}} q_{ij}.$$
By Lemma \ref{avg}, $\overline{a_{ij}} \in \mcP^K$, and by induction on dimension, we can assume the $\overline{a_{ij}}$ lie in the algebra generated by the $q_{ij}$.   This completes the proof of finite generation.

To analyze the quotient $X/K$, let $p_1, \ldots, p_M$ be a finite generating set for the algebra $\mcP^K$, and consider the polynomial mapping $p = (p_1, \ldots, p_m)\co \bbR^N \to \bbR^M$.  The image of $X$ under this mapping is a semi-algebraic subset of $\bbR^M$, and we claim that the induced map $\pi\co \bbR^N/K\to \bbR^M$ is a homeomorphism onto its image.  To show that $\pi$ is injective, it suffices to check that for any two distinct orbits $K\cdot x$ and $K\cdot y$ there exists $p\in \mcP^K$ with $p(x) \neq p(y)$.  This follows by applying Urysohn's Lemma, the Weierstrass approximation theorem, and Lemma \ref{avg}.

When $X \subset \bbR^N$ is compact, it follows immediately that the injection $\pi\co X/K\to \bbR^M$ is a homeomorphism onto its image.  Schwarz \cite{Schwarz-smooth} shows that the full map $\bbR^N/K \to \bbR^M$ is a homeomorphism onto its image; we only need the compact case.


\section{Crystallographic groups} $\label{crystal}$

An (abstract) crystallographic group $\Gamma$ is a group admitting an embedding into the group of isometries of Euclidean space $\bbR^k$, such that the image of $\Gamma$ in $\Isom(\bbR^k)$ is discrete and $\bbR^k/\Gamma$ is compact.  We will call the resulting action of $\Gamma$ on $\bbR^k$ a \emph{crystallographic action} of dimension $k$.  As explained below, all crystallographic actions of a fixed $\Gamma$ have the same dimension, which we will refer to as the dimension of $\Gamma$.
Standard references for the theory of crystallographic groups are Ratcliffe \cite[\S7.4]{Ratcliffe} and Wolf \cite[Chapter 3]{Wolf}.
Surprisingly, there are only finitely many crystallographic groups in each dimension (this was part of Hilbert's $18^{\textrm{th}}$ problem, solved by Bieberbach).  We note that every discrete subgroup of $\Isom(\bbR^k)$ acts crystallographically on some affine subspace of $\bbR^k$ \cite[Theorem 5.4.6]{Ratcliffe} and hence surjects onto a crystallographic group, but this action may have a non-trivial kernel.

There is an isomorphism of topological groups $\bbR^k \rtimes O(k) \srt{\homeo} \Isom (\bbR^k)$, where the semi-direct product has the product topology \cite[\S1.3 and Theorem 5.2.4]{Ratcliffe}.  Here $\bbR^k$ and $O(k)$ act on $\bbR^k$ via their natural actions (by translation and matrix multiplication, respectively), so any $(a, A) \in \bbR^k \cross O(k)$ acts on $\bbR^k$ via the isometry
$$v\goesto a + Av.$$
The semi-direct product decomposition gives us a map $\pi\co \Isom(\bbR^k) \maps O(k)$, whose kernel is precisely the subgroup of translations.  Each subgroup $\Gamma < \Isom(\bbR^k)$ then has a corresponding map to $O(k)$ whose kernel consists of all $\gamma\in \Gamma$ which act on $\bbR^k$ by translation.  We call this normal subgroup $A$ the \emph{translation subgroup} of $\Gamma$, and we call $\Gamma/A$ the \emph{point group} of $\Gamma$.  Note that $A$ is always abelian.

If $\Gamma$ has a crystallographic action of dimension $k$, then its translation subgroup $A\normal \Gamma$ is in fact a free abelian group of rank $k$, and $[\Gamma : A] < \infty$ \cite[Theorem 7.4.2]{Ratcliffe}, making $\Gamma$ virtually $\bbZ^k$.
Moreover, $\Gamma$ sits in an extension 
\begin{equation} \label{extn} 1\maps A\maps \Gamma \maps Q \maps 1\end{equation}
with $A\isom \bbZ^k$ and $Q$ finite.  The following result, also part of \cite[Theorem 7.4.2]{Ratcliffe}, is helpful in recognizing crystallographic actions.

\begin{proposition}$\label{full-rk}$ If $\Gamma < \Isom (\bbR^k)$ is discrete and its translation subgroup has rank $k$, then $\Gamma$ is crystallographic; that is, $\bbR^k/\Gamma$ is compact.
\end{proposition}

In any extension of the form (\ref{extn}), conjugation in $\Gamma$ induces an action of $Q$ on $A$.  When $\Gamma$ is crystallographic, this action is faithful.  Indeed, if $\gamma \in \Gamma$ acts trivially on $A$, then the subgroup $A'$ generated by $A$ and $\gamma$ is an abelian, discrete subgroup of $\Isom(\bbR^k)$, and Ratcliffe's analysis of such subgroups \cite[Theorem 5.4.4]{Ratcliffe} shows that $A'$ acts by translations on all of $\bbR^k$.  Hence $\gamma \in A$.  Note that this also shows that $A$ is a maximal abelian subgroup of $\Gamma$.
(By \cite[Theorem 7.4.5]{Ratcliffe},  every group $\Gamma$ sitting in an extension $\bbZ^k \injects \Gamma \surjects Q$, with $Q$ a finite group acting faithfully on $A$, actually admits a $k$--dimensional crystallographic action.)

We will work with abstract groups, rather than with groups  with a chosen crystallographic action, so it is helpful to know that the translation subgroup and the point group can be defined without reference to such an action.  For completeness, we give a proof.  A similar result is proven in Wolf~\cite[Theorem 3.2.9]{Wolf}. 

\begin{lemma} $\label{translations}$ Let $\Gamma$ denote an abstract crystallographic group.  Then there is exactly one subgroup of $A\leqs \Gamma$ which has finite index and is maximal abelian in $\Gamma$.  This subgroup is a finite rank, free abelian group, and is normal in $\Gamma$.  Moreover, in any crystallographic action of $\Gamma$, $A$ is mapped onto the subgroup of translations.
\end{lemma}  
\begin{proof} Say $\Gamma$ has a crystallographic action on $\bbR^k$, with $A$ the corresponding subgroup of translations.  Then we have seen that $A$ is a normal, finite index subgroup isomorphic to $\bbZ^k$, and that $A$ is maximal abelian in $\Gamma$.

To complete the proof, we will show that the translation subgroup $A < \Gamma$ contains all finite index abelian subgroups $A' < \Gamma$.  Any such subgroup $A'$ is finitely generated, since it contains the free abelian group $A\cap A'$ as a finite index subgroup.  
In general, if $A_1$ and $A_2$ are finite index, finitely generated abelian subgroups of a group $G$, then the finitely generated abelian group $A_1\cap A_2$ has finite index in both $A_1$ and $A_2$, so $\rk (A_1) = \rk (A_1\cap A_2) = \rk (A_2)$. 
Hence $A'$ has  rank $k$. 
It follows from Ratcliffe's discussion of discrete abelian subgroups of $\Isom (\bbR^k)$ \cite[Theorem 5.4.4]{Ratcliffe} that $A'$ must in fact be free abelian, and must act by translations on $\bbR^k$.
Thus $A'$ lies inside the translation subgroup $A$. 
\end{proof}


 \section{Spaces of representations}$\label{rep-spaces}$

Let $\Gamma$ be a finitely generated discrete group. 
In this section we introduce basic terminology and facts regarding spaces of unitary representations of $\Gamma$.  We denote the set of homomorphisms $\rho\co \Gamma\to U(n)$ by $\Hom(\Gamma, U(n))$.  This space is naturally topologized using the product topology on $U(n)^\Gamma$ (which is the same as the compact-open topology on $\Map(\Gamma, U(n)) \homeo U(n)^\Gamma$).  One may check that $\Hom(\Gamma, U(n))$ is closed in $U(n)^\Gamma$, hence compact.  
A generating set $S\subset \Gamma$ determines an injection from
$\Hom(\Gamma, U(n))$ into a product of $|S|$ copies of $U(n)$, and by compactness this map is always a homeomorphism onto its image.  If our generating set contains $m < \infty$ elements, then the corresponding embedding gives $\Hom(\Gamma, U(n))$ the structure of a real algebraic variety, cut out from the real algebraic variety $U(n)^m$ by the relations in $\Gamma$.  For our purposes, it will not be necessary to consider the relationship between the algebraic structures induced by different generating sets.  By convention, we set $U(0) = \{1\}$ so that $\Hom(\Gamma, U(0))$ is the one-point space.

The block sum maps $U(n) \cross U(m) \to U(n+m)$, which we denote $(A, B) \goesto A\oplus B$, determine corresponding block sum maps on representation spaces, which we again denote $(\rho, \psi) \goesto \rho\oplus \psi$.  By convention, block sum with the zero-dimensional matrix or representation is the identity map.  The action of the unitary group on itself by conjugation induces an action of $U(n)$ on $\Hom(\Gamma, U(n))$.

\begin{definition} The \emph{moduli space of} $U(n)$--\emph{representations} is the quotient space 
$$\bHom_n (\Gamma) \mathrel{\mathop :}=\Hom(\Gamma, U(n))/U(n).$$  

We say that $\rho \in \Hom(\Gamma, U(n))$ is reducible if $\rho\isom \phi\oplus \psi$ for some positive-dimensional representations $\phi, \psi$; otherwise we say that $\rho$ is irreducible.
Let 
$$ \Sum(\Gamma, U(n)) \subset \Hom(\Gamma, U(n)) \textrm{ and } \Irr (\Gamma, U(n)) \subset \Hom(\Gamma, U(n))$$ 
denote the subspaces of reducible and of irreducible representations (respectively).  The \emph{moduli space of irreducible} $U(n)$--\emph{representations} is the quotient space 
$$\bIrr_n (\Gamma) \mathrel{\mathop :}= \Irr(\Gamma, U(n))/U(n).$$  
Similarly, we denote $\Sum(\Gamma, U(n))/U(n)$ by $\bSum_n (\Gamma)$.
\end{definition}

By Corollary \ref{bHom-sa}, we can view $\bHom_n (\Gamma)$ as a semi-algebraic subset of Euclidean space, in such a way that the quotient map $\Hom(\Gamma, U(n)) \to \bHom_n (\Gamma)$ is semi-algebraic.  We will take this viewpoint from now on.

\begin{lemma}$\label{loc-cpt}$ For any discrete group $\Gamma$, $\bSum_n (\Gamma) \subset \bHom_n (\Gamma)$ and $\Sum (\Gamma, U(n)) \subset \Hom(\Gamma, U(n))$ are closed and semi-algebraic.  
Hence $\Irr(\Gamma, U(n)) \subset \Hom(\Gamma, U(n))$ and 
$\bIrr_n (\Gamma) \subset \bHom_n (\Gamma)$ are open, semi-algebraic subsets.
\end{lemma}
\begin{proof}  
To show that $\Sum(\Gamma, U(n)) \subset \Hom(\Gamma, U(n))$ is semi-algebraic, it will suffice to show that its image in $\bHom_n (\Gamma)$ is semi-algebraic, because the inverse image of a semi-algebraic subset under a semi-algebraic map remains semi-algebraic.  The set of reducibles in $\bHom_n(\Gamma)$ is the union, over $k$, of the images of the block sum maps
\begin{equation}\label{bsm}
\Hom(\Gamma, U(k)) \cross \Hom(\Gamma, U(n-k)) \maps \Hom(\Gamma, U(n)) \maps \bHom_n (\Gamma).
\end{equation}
Since both the block sum map and the quotient map are semi-algebraic, so is their composite (see Section \ref{real}).  We have now written $\bSum_n (\Gamma)$ as the union of a finite collection of semi-algebraic subsets, and hence $\bSum_n(\Gamma)$ is itself semi-algebraic.

To see that $\Sum(\Gamma, U(n)) \subset \Hom(\Gamma, U(n))$ is closed, note that since the domains of the block-sum maps (\ref{bsm}) are compact, so is the union of their images.  Hence $\bSum_n (\Gamma)$ is closed, and so is its inverse image $\Sum(\Gamma, U(n))$.  
The final statement follows from the fact that complements of semi-algebraic sets are semi-algebraic.
\end{proof}

Lawson's articles \cite{Lawson-prod-form} and \cite{Lawson-simul} use different notations for 
$\bHom_n (\Gamma) /  \bSum_n(\Gamma),$
which contains the moduli space of irreducibles $\bIrr_n (\Gamma)$ as  the complement of the basepoint.
The following observation motivates the  notation $\bIrrp_n (\Gamma)$ for this space.

\begin{lemma}$\label{one-pt2}$ The space $\bHom_n (\Gamma)$ is compact Hausdorff, and
\begin{equation*}
\begin{split}
\bIrrp_n (\Gamma)  \homeo \big(\Hom(\Gamma, U(n))/\Sum(\Gamma, U(n))\big)/U(n)
\end{split}
\end{equation*}
is the one-point compactification of the moduli space of irreducible representations.
\end{lemma} 
\begin{proof}  
The space $\bIrrp_n (\Gamma)$ is a compact Hausdorff space, because it is the quotient of the compact semi-algebraic set $\bHom_n (\Gamma)$ by the closed subspace $\bSum_n (\Gamma)$. 
The moduli space $\bIrr_n (\Gamma)$ embeds in $\bIrrp_n (\Gamma)$ as the complement of the basepoint, and every compact Hausdorff space $X$ is the one-point compactification of $X - \{x\}$ for each $x\in X$ (see, for example, Munkres~\cite[Theorem 29.1]{Munkres}).
\end{proof}

\begin{remark} The fact that $\bIrrp_n (\Gamma)$ is Hausdorff can be proven using basic facts about quotients by compact groups (Munkres~\cite[\S31, Exercises 6,7,8]{Munkres}).  All spaces encountered in this paper are Hausdorff, as can be shown using similar methods.
\end{remark}

We note a simple fact which is extremely helpful in computations.

\begin{proposition}$\label{max-dim}$ If $\Gamma$ contains an abelian subgroup $A$ of index $m <  \infty$, then all irreducible unitary representations of $\Gamma$ have dimension at most $m$.\end{proposition}

This is proven in Serre \cite[\S3.1]{Serre}.  The key point is that simultaneously commuting diagonalizable (e.g. unitary) matrices are simultaneously diagonalizable, so irreducible representations of $A$ are 1--dimensional (this fact will be used several times).  The result follows by restricting an irreducible representation $\rho\co \Gamma\to U(n)$ to  $A$, and noting that if $S$ is a set of coset representatives for $\Gamma/A$ and $W\subset \bbC^n$ is an irreducible summand of $\rho|_A$, then the translates $sW$, $s\in S$, generate $\bbC^n$.


\section{Induced Representations and Projective Representations}$\label{ind-proj}$

As discussed in the introduction, we are considering groups $\Gamma$ with a finite index subgroup $H\isom \bbZ^k$.  The normal core $A = \bigcap_\gamma \gamma H\gamma^{-1}$ is a \emph{normal} finite index subgroup, and $A\isom \bbZ^k$.  So such groups are finite extensions of $\bbZ^k$.

The overall structure of our arguments is based on the following result from representation theory (see Serre~\cite[Proposition 24]{Serre}).  

\begin{theorem} $\label{Serre8.1}$ Let $A$ be an abelian normal subgroup of finite index in a discrete group $\Gamma$.  For every irreducible representation $\rho \co \Gamma \to U(n)$, either
\begin{itemize}
\item $\rho$ is isomorphic to $\Ind_H^\Gamma (\rho')$ for some proper subgroup $H<\Gamma$ containing $A$ and some irreducible unitary representation $\rho'$ of $H$, or 

\item the restriction of $\rho$ to $A$ is scalar (a direct sum of isomorphic $1$--dimensional representations).
\end{itemize}
\end{theorem}

This result is proven by letting $H$ be the stabilizer of an isotypic component of $\rho$.  Serre's book focuses on finite groups, but the proof extends without change to (unitary) representations of infinite discrete groups, using the fact (Proposition \ref{max-dim}) that all irreducible unitary representations of $A$ are 1--dimensional, along with Schur's Lemma, which implies that \emph{unitary} representations of infinite discrete groups admit unique decompositions into irreducible summands

Theorem~\ref{Serre8.1} shows that if $\Gamma$ sits in an extension
$$1\maps A \maps \Gamma \maps Q\maps 1$$
with $A$ abelian and $Q$ finite, then
every representation $\rho\co \Gamma\to U(n)$ is either 
\begin{itemize}
\item induced from a some $H<\Gamma$ with $A\leqs H$, or else 
\item produces a homomorphism $\rho\co Q = \Gamma/A \to \PU(n)$, 
\end{itemize}
where $\PU(n)$, the projective unitary group, denotes the quotient of $U(n)$ by the subgroup of scalar matrices $\lambda I$, $\lambda\in S^1$.  (From now on, we will denote the scalar subgroup simply by $S^1 \subset U(n)$.)  In Case (1), the group $H$ sits in an extension
$$1 \maps A \maps H \maps H/A \maps 1,$$
so $H$ still satisfies the hypotheses of our theorems, and this will allow for an induction argument based on the order of the quotient group $Q$.
We will analyze these two classes of representations (induced and projective) separately in the following sections, and then combine our results to prove the main theorems.


\section{Induction}$\label{ind}$

In this section, $\Gamma$ will denote a finitely generated discrete group with a subgroup $H < \Gamma$ of index $m<\infty$.  By the Schreier Index Formula, $H$ is also finitely generated.

We need to analyze the induction maps 
$$\Ind_H^\Gamma\co \bHom_n (H) \maps \bHom_{nm} (\Gamma).$$
Abstractly, induction can be defined by viewing a representation $\rho \co H\to U(n)$ as a (left) module $V$ over the group ring $\bbC[H]$, and defining 
$\Ind_H^\Gamma (V)$ to be the (left) $\bbC[\Gamma]$--module $\bbC[\Gamma] \otimes_{\bbC[H]} V$.  
We need a description of induction as a continuous function between spaces of unitary matrices.  
This is easily obtained by choosing a set of coset representatives $\gamma_1, \ldots, \gamma_m$ for $\Gamma/H$, yielding a direct sum decomposition
$$\bbC[\Gamma] \isom \gamma_1 \bbC[H] \oplus \cdots \oplus \gamma_m \bbC[H]$$
of $\bbC[\Gamma]$ as a \emph{right} $\bbC[H]$--module.  A representation $\rho\co H\to U(n)$ gives $\bbC^n$ a left $\bbC[H]$--module structure, and we denote this module by $\bbC^n_\rho$.  We now have an ordered basis for the complex vector space $\bbC[\Gamma] \otimes_{\bbC[H]} \bbC^n_\rho$, given by
$$\gamma_1 \otimes e_1, \gamma_1\otimes e_2, \cdots, \gamma_1 \otimes e_n, \gamma_2\otimes e_1, \cdots, \gamma_2\otimes e_n, \cdots, \gamma_m\otimes e_1, \cdots, \gamma_m \otimes e_n,$$
where the $e_i$ are the standard basis vectors in $\bbC^n$.
Endowing $\bbC[\Gamma] \otimes_{\bbC[H]} \bbC^n_\rho$ with the unique Hermitian metric making this basis orthonormal, 
one checks that $\Gamma$ acts isometrically on 
$\bbC[\Gamma] \otimes_{\bbC[H]} \bbC^n_\rho$ yielding a well-defined, continuous induction map 
\begin{equation}\label{ind-map}\Ind_H^\Gamma \co \Hom(H, U(n)) \maps \Hom(\Gamma, U(n)),\end{equation}
The $(i,j)$--entry of $\Ind_H^\Gamma (\rho) (\gamma)$ is an entry of $\rho(h)$, for some $h$ depending only on $\gamma$, $i$, and $j$.  It follows that $\Ind_H^\Gamma$ is an \emph{algebraic map}, that is, its graph is an algebraic subset of $\Hom(H, U(n)) \cross \Hom(\Gamma, U(n))$.

Defining $\bInd_H^\Gamma ([\rho]) = [\Ind_H^\Gamma (\rho)]$, we obtain a continuous map
\begin{equation}\label{ind-map2}\bInd_H^\Gamma\co \bHom_n (H) \maps \bHom_{nm} (\Gamma),\end{equation}
(Note that if  $\rho' = X\rho X^{-1}$, then  
$\Ind_H^\Gamma(\rho') = (I_m\otimes X) \big(\Ind_H^\Gamma (\rho)\big) (I_m\otimes X)^{-1}$, where $I_m\otimes X$ denotes the  $m$--fold block sum of $X$.  Hence $\bInd_H$ is well-defined).  

\begin{remark} Although the maps (\ref{ind-map}) depend on our chosen isomorphism $\bbC[\Gamma]\otimes_{\bbC[H]} \bbC^n \isom \bbC^{nm}$, the maps (\ref{ind-map2}) are independent of this choice.  This follows from the fact that if two unitary representations are conjugate via a general linear matrix, then they are in fact conjugate via a unitary matrix.  (A more general statement is proven in Proposition \ref{inv-metric} below.)
\end{remark}

Letting $b\co U(n) \to U(nm)$ denote the block-sum map $X \goesto mX$, we conclude from the above discussion that $\Ind_H^\Gamma$ is equivariant with respect to the conjugation actions of $U(n)$ and $U(nm)$, in the sense that $\Ind_H^\Gamma (X\cdot \rho) = b(X) \cdot \Ind_H^\Gamma (\rho)$.   

\begin{proposition} $\label{Ind-sa}$ 
The map
$\bInd_H^\Gamma \co \bHom_n (H) \maps \bHom_{nm} (\Gamma)$
is semi-algebraic.
\end{proposition}
\begin{proof}  By definition, we need to show that the graph of $\bInd_H^\Gamma$ is a semi-algebraic subset of $\bHom_n (H) \cross \bHom_{nm} (\Gamma)$.
We observed above that the graph 
$$\Gr (\Ind_H^\Gamma) \subset \Hom(H, U(n)) \cross \Hom(\Gamma, U(nm))$$
is algebraic.  Moreover, the map
\begin{equation}\label{q-prod} \Hom(H, U(n)) \cross \Hom(\Gamma, U(nm)) \maps \bHom_n (H) \cross\bHom_{nm} (\Gamma)
\end{equation}
is a product of semi-algebraic maps, so it too is semi-algebraic.

Since $\Gr(\bInd_H^\Gamma)$ is the image of $\Gr(\Ind_H^\Gamma)$ under the semi-algebraic map (\ref{q-prod}), the Tarski--Seidenberg Theorem implies that $\Gr(\bInd_H^\Gamma)$ is semi-algebraic, as desired.
\end{proof}

\begin{definition}$\label{ind-def}$
Let $\bHom_{nm} (\Gamma)_H \subset \bHom_{nm} (\Gamma)$ denote the image 
$\bInd_H^\Gamma (\bIrr_n (H))$.
If $m = [\Gamma : H]$ does not divide $k$, then no $k$--dimensional representations of $\Gamma$ are induced from $H$, and we set $\bHom_{k} (\Gamma)_H = \emptyset$.
\end{definition}

Since the subspace of irreducible representations is semi-algebraic (Lemma \ref{loc-cpt}) and the induction map is a semi-algebraic mapping (Proposition \ref{Ind-sa}), we obtain the following result as a consequence of the Tarski-Seidenberg Theorem.

\begin{corollary}$\label{Ind_H}$ The subspace $\bHom_{n} (\Gamma)_H$ is a semi-algebraic subset of $\bHom_{n} (\Gamma)$. 
\end{corollary}

In Section \ref{stable-moduli}, we need to consider the interaction between induction and tensor products.  A construction similar to that  for induction yields continuous maps
$$\Hom(\Gamma, U(n)) \cross \Hom(\Gamma, U(m)) \srm{\otimes} \Hom(\Gamma, U(nm)),$$
which depend on a choice of ordered basis for the vector space $\bbC^n \otimes \bbC^m$, and
which descend to continuous maps on the moduli spaces.  These latter maps are independent of the choices made, and in terms of $\bbC[G]$--modules, they send a pair of modules $V$ and $W$ to the module $V\otimes_{\bbC} W$, with the diagonal action of $\bbC[G]$.  For $\rho\co \Gamma\to U(n), \psi\co H\to U(k)$, we   have the Projection Formula (Serre \cite[\S 7.2]{Serre})
\begin{equation}  \label{projection}   
\Ind_H^\Gamma \left(  (\Res^\Gamma_H \rho) \otimes \psi \right) \isom \rho \otimes \Ind_H^\Gamma (\psi). 
\end{equation}


\section{Projective representations}$\label{proj}$

In this section, $\Gamma$ will denote a finitely generated discrete group, and $A\normal \Gamma$ will denote a finite index normal subgroup with quotient $Q = \Gamma/A$.

Recall from Section \ref{ind-proj} that if $A$ is abelian, each irreducible representation $\rho\co \Gamma\to U(n)$ is either induced from a subgroup containing $A$, or satisfies $\rho(A) \subset S^1$.  In the latter case, we obtain a representation 
$\brho \co Q= \Gamma/A \to \PU(n) = U(n)/S^1$ (which we will avoid referring to as the ``induced representation").

\begin{definition} Let $\rHom$ denote the subspace of $\Hom(\Gamma, U(n))$ consisting of representations $\rho$ such that $\rho(A) \subset S^1$.  
\end{definition}  

\begin{lemma} $\label{proj-subvar}$ If $\Gamma$ is finitely generated, then $\rHom$ is a closed, $U(n)$--invariant subvariety of $\Hom(\Gamma, U(n))$.  In particular, $\rHom$ is compact, and its image in $\bHom_n (\Gamma)$ is a closed semi-algebraic subset.
\end{lemma} 
\begin{proof} Invariance under $U(n)$ follows from the fact that $S^1$ is central in $U(n)$.  The subspace $\rHom\subset\Hom(\Gamma, U(n))$  is the closed subvariety cut out by the requirements that for each $a \in A$, all the off-diagonal entries of $\rho(a)$ are zero and all  the diagonal entries are equal to one another.   The final statement follows from the Tarski--Seidenberg Theorem and  Corollary \ref{bHom-sa}.
\end{proof}

Given a discrete group $G$, the space $\Hom(G, \PU(n))$ has the subspace topology from the product space $\PU(n)^G$, and has a natural action of $\PU(n)$ by conjugation.  
Using Schur's theory of projective representations of finite groups, we will analyze the irreducible representations in $\rHom$ in terms of the map $\rHom \to \Hom(Q, \PU(n))$.  We will now set up some basic terminology regarding projective representations.  For more detail about this subject, we refer the reader to Karpilovsky \cite[Chapter 3]{Karpilovsky}.

Homomorphisms $G\to \PU(n)$ are closely connected to what are usually called projective unitary representations, that is, functions $\rho\co G\to U(n)$ such that $\rho(g_1) \rho(g_2)= \sigma(g_1, g_2) \rho(g_1 g_2)$ for some scalars $\sigma(g_1, g_2)\in S^1$ (we assume that $\rho(1) = 1$).  
Replacing $U(n)$ by $\GL_n (\bbC)$ and $S^1$ by $\bbC^*$ yields the notion of a projective linear representation.
Each homomorphism $G\to \PU(n)$ may be lifted (in many ways) to a projective unitary representation, and conversely each projective representation defines a homomorphism $G\to \PU(n)$. 

There are four notions of equivalence between projective unitary representations $\rho\co G\to U(n)$ and $\rho'\co G\to U(n)$.  We say that $\rho$ and $\rho'$ are \emph{projectively equivalent} if there exists a matrix $A\in \GL_n (\bbC)$ such that for each $g\in G$, $A \rho (g) A^{-1} \rho'(g)^{-1}\in \bbC^*$.  We write $\rho\sim_{\GL} \rho'$ in this situation.  By requiring the matrix $A$ to lie in  $U(n)$, we obtain a (potentially) stricter notion of equivalence, \emph{projective unitary equivalence}, which we denote by $\rho\sim_{U} \rho'$.  Next, we say that $\rho$ and $\rho'$ are \emph{linearly equivalent} ($\rho \approx_{\GL} \rho'$) if there exists a matrix $A\in \GL_n (\bbC)$ such that $A \rho (g) A^{-1} = \rho'(g)$ for all $g\in G$.  Finally, we obtain the notion of \emph{unitary linear equivalence} ($\rho \approx_U \rho'$) by requiring the matrix $A$ to lie in $U(n)$.  

Given two homomorphisms $\phi, \psi\co G\to PU(n)$, we write $\phi\isom \psi$ if there exists $A\in PU(n)$ with $A\phi A^{-1} = \psi$.  The reader may check that such isomorphism classes of homomorphisms $G\to \PU(n)$ correspond bijectively with \emph{projective unitary equivalence classes} of projective unitary representations $G\to U(n)$.  

We say that a projective representation $\rho\co G \to U(n)$ is \emph{irreducible} if there is no proper, non-zero subspace of $\bbC^n$ that is invariant under each of the matrices $\rho(g)$, $g\in G$.
We say that a homomorphism $G\to \PU(n)$ is irreducible if one of its lifts to a projective representation is irreducible.  Note that if $\rho$ and $\rho'$ are two such lifts, then for each $g\in G$, $\rho(g)$ and $\rho'(g)$ differ by a scalar, and hence the invariant subspaces for $\rho$ and $\rho'$ are the same.  Thus \emph{all} projective representations associated to an irreducible homomorphism $G\to \PU(n)$ are irreducible.
On the other hand, if $\rho\co G\to U(n)$ is an irreducible projective representation, then the corresponding homomorphism $\brho\co G\to \PU(n)$ is irreducible as well, since $\rho$ itself is a lift of $\brho$.

We note that irreducibility is preserved under projective equivalence: if $\rho\co G\to U(n)$ is irreducible and $\rho\sim_{\GL} \rho'$, then $\rho'$ is irreducible as well.  In particular, a homomorphism $\rho\co G\to \PU(n)$ is irreducible if and only if all of its $\PU(n)$--conjugates are irreducible.

We note another important fact regarding irreducible projective representations.

\begin{lemma}$\label{irred-proj}$ A representation $\rho\in \rHom$ is irreducible if and only if $\brho\co Q \to \PU(n)$ is irreducible.
\end{lemma}
\begin{proof}
Choose a set of coset representatives $\{\gamma_i\}_i$ for the cosets of $A$, and lift $\brho$ to a projective representation $\wt{\rho} \co Q \to U(n)$ by setting $\wt{\rho} ([\gamma_i]) = \rho(\gamma_i)$.  Note that for each $\gamma\in \Gamma$, we have $\gamma = \gamma_i a$ for some $a\in A$ and some $i$, so $\rho(\gamma) = \rho(\gamma_i a) = \wt{\rho}([\gamma_i]) \rho(a)$.  Since $\rho(A) \subset S^1$, we see that each matrix comprising the representation $\rho$ is a scalar multiple of a matrix appearing in $\wt{\rho}$, and of course each matrix appearing in $\wt{\rho}$ also appears in $\rho$.  Hence $\wt{\rho}$ and $\rho$ have the same invariant subspaces, so $\wt{\rho}$ is irreducible if and only if $\rho$ is irreducible.
\end{proof}

Our next goal is to show that for \emph{irreducible} projective unitary representations $\rho\co G\to U(n)$, the equivalence relations $\approx_\GL$ and $\approx_U$ coincide, and if we further assume that $G$ is finite, so do the relations $\sim_U$ and $\sim_\GL$.

\begin{lemma}$\label{1-dim}$ Let $\rho\co G\to \GL(V)$ be an irreducible projective representation of $G$ on a finite-dimensional complex vector space $V$.  If $\la\, , \, \ra$ and $\la \, , \, \ra'$ are two $\rho$--invariant Hermitian metrics on $V$, (i.e. $\la \rho(g) v, \rho(g) w\ra = \la v, w \ra$ for all $g\in G$, and similarly for $\la \, , \, \ra'$) then there exists a scalar $t\in \bbR^+$ such that 
$$\la v , w \ra' = t \la v , w \ra$$
for all $v, w\in V$.
\end{lemma}  
\begin{proof} Let $V^*$ denote the  space  of \emph{conjugate}--linear functionals on $V$; that is
$$V^* = \{f\co V\to \bbC \, | \, \forall \, v, w\in V, \, \lambda \in \bbC, \, \, f(v + \lambda w) = f(v) + \overline{\lambda} f(w)\}.$$
Note that $V^*$ is a complex vector space under point-wise addition and (ordinary) scalar multiplication.  Any Hermitian metric $\la \, , \, \ra$ on $V$ induces a complex-linear isomorphism $\phi\co V\srt{\isom} V^*$, where $\phi(v) = \phi_v \co V\to\bbC$ is given by the formula
$$\phi_v (w) = \la v, w\ra \in \bbC.$$

Now, any projective representation $\rho\co G \to \GL(V)$ induces a projective representation $\rho^* \co G \to \GL(V^*)$ (called the contragradient of $\rho$).  If we write the resulting actions of $G$ on $V$ and $V^*$ as $\rho(g) v = g\cdot v$ and $\rho^* (g) f = g\cdot f$, then $\rho^*$ is defined by the formula
$$(g\cdot f) (v)  = f(g^{-1} \cdot v).$$
One may now check that $g\cdot f$ is conjugate linear, and that,  up to multiplication by scalars, this formula gives an action of $G$ on $V^*$.  
Hence $\rho^*$ is a projective representation of $G$.

Next, say we have two $\rho$--invariant metrics $\la \, , \, \ra$ and $\la \, , \, \ra'$ on $V$.  Then a short computation shows that the resulting isomorphisms $\phi, \phi'\co V\to V^*$ are $G$--equivariant, with respect to the representations $\rho$ and $\rho^*$.  Hence the linear isomorphism 
$$X = \phi^{-1} \phi'\co V\to V$$
satisfies $X^{-1} \rho X = \rho$.  Since $\rho$ is irreducible, the usual proof of Schur's lemma shows that $X = t I$ for some $t \in \bbC$.  Now for any $v\in V$, $\phi' (v) = \phi (t v) = t \phi (v)$, so
$$\la v, w \ra' = t \la v, w \ra$$
for any $w\in V$.  All that remains is to check that $t \in \bbR^+$; this follows from the assumption that $\la \, , \, \ra$ and $\la \, , \, \ra'$ are (positive definite) Hermitian metrics.
\end{proof}

\begin{proposition} $\label{inv-metric}$ If $\rho, \rho'\co G \to U(n)$ are irreducible projective unitary representations, then $\rho\approx_\GL \rho' \iff \rho \approx_U \rho'$.
\end{proposition}
\begin{proof} The direction $\rho\approx_U \rho' \implies \rho \approx_\GL \rho'$ is immediate.  
For the converse, let $ P\in \GL_n (\bbC)$ be a matrix satisfying $P^{-1} \rho(g) P = \rho'(g)$ for all $g\in G$.  Define a Hermitian metric on $\bbC^n$
by setting $\la v, w \ra_P = \la Pv, Pw\ra$, where $\la \, , \, \ra$ is the standard Hermitian metric on $\bbC^n$.  Then $\la \, , \, \ra_P$ is an invariant metric for the representation $\rho'$, because
\begin{equation*} \begin{split}
\la \rho'(g) v, \rho'(g) w\ra_P  & = \la P \rho'(g) v, P \rho'(g) w \ra = \la  \rho(g) P v,  \rho(g) P w \ra \\
& = \la  P v,  P w \ra = \la v, w \ra_P,
\end{split}
\end{equation*}
with the third equality following from the fact that $\rho(g) \in U(n)$.  

Now, both $\la \, , \, \ra$ and $\la \, , \, \ra_P$ are $\rho'$--invariant.  By Lemma~\ref{1-dim}, it follows that $\la \, , \, \ra_P = t \la \, , \, \ra$ for some $t\in \bbR^+$.  The matrix $U = \frac{1}{\sqrt{t}} P$ still satisfies
$U^{-1} \rho U = \rho'$, and $U\in U(n)$ because
\begin{equation*} 
\begin{split}\la Uv, Uw \ra & = \la \frac{1}{\sqrt{t}} Pv, \frac{1}{\sqrt{t}} Pw \ra = \frac{1}{ t } \la Pv, Pw \ra \\
& =   \frac{1}{ t } \la  v,  w \ra_P =  \frac{1}{ t } \left( t \la  v,  w \ra \right) = \la v, w \ra
\end{split}
\end{equation*}
for any $v, w\in \bbC^n$.
\end{proof}

\begin{corollary} $\label{U-GL}$ If $\rho, \rho'\co G \to U(n)$ are irreducible projective unitary representations of a finite group $G$, then $\rho\sim_\GL \rho' \iff \rho \sim_U \rho'$.
\end{corollary}
\begin{proof}  Again, the direction $\rho\sim_U \rho' \implies \rho \sim_\GL \rho'$ is immediate, so we assume that $\rho \sim_\GL \rho'$.  This means that for some matrix $P\in \GL_n (\bbC)$ and some function $\lambda\co G\to \bbC^*$, we have
\begin{equation}\label{simGL}P\rho(g) P^{-1} = \lambda(g) \rho'(g)\end{equation}
 for all $g\in G$.  
 
 We claim that $\lambda(g) \in S^1$ for each $g\in G$.  If $g$ has order $m$, then raising both sides of (\ref{simGL}) to the $m^\textrm{th}$ power shows that $P\rho(g)^m P^{-1} = \lambda(g)^m \rho'(g)^m$.  Since $\rho$ and $\rho'$ are projective unitary representations, we have $\rho(g)^m, \rho'(g)^m \in S^1$, so $P \rho(g)^m P^{-1} = \rho(g)^m$ and hence $\lambda(g)^m = \rho(g)^m/\rho'(g)^m \in S^1$.  This implies that $\lambda(g)$ itself lies in $S^1$, as desired.

Now, setting $\rho''(g) = \lambda(g) \rho'(g)$, we see that $\rho''(g) \co G\to U(n)$ is still a projective \emph{unitary} representation, and now $\rho \approx_\GL \rho''$ (via the matrix $P\in \GL_n (\bbC)$).  By Proposition \ref{inv-metric}, there exists a matrix $U\in U(n)$ with
$U\rho(g) U^{-1} = \rho''(g) = \lambda(g) \rho'(g)$ for all $g\in G$, so $\rho(g) \sim_U \rho'(g)$, as desired.
\end{proof}

Using Corollary~\ref{U-GL}, we will deduce a key finiteness result for projective unitary representations of \emph{finite} groups.
This will be a corollary of the following classical result due to Schur~\cite{Schur} (see also Karpilovsky~\cite{Karpilovsky}, Tappe~\cite[Corollary 3.6]{Tappe}).

\begin{theorem}$\label{schur}$  For any finite group $G$, the number of projective equivalence classes ($\sim_{\GL}$--classes) of irreducible projective representations $G\to \GL_n (\bbC)$ is finite. 
\end{theorem}

We sketch the proof.  
Each projective representation $\rho\co G\to \GL_n (\bbC)$ has an associated cohomology class in $H^2 (G ; \bbC^*)$.  Specifically, given $g, h\in G$, we have $\sigma(g,h) \mathrel{\mathop :}= \rho(g) \rho(h) \rho(gh)^{-1} \in \bbC^*$, and the assignment $(g,h) \goesto \sigma(g,h)$ is a $\bbC^*$--valued $2$--cocycle on $G$.  If $\rho\sim_{GL} \rho'$, then the corresponding cocyles are cohomologous \cite[Chapter 3, Lemma 1.1 (i)]{Karpilovsky}, so we have a well-defined class in $H^2(G; \bbC^*)$ associated to each $\sim_\GL$--class of projective representations.  Now, for any finite group $G$, it turns out that the group $H^2 (G; \bbC^*)$ is finite (for a short proof, see \cite[Chapter 2, Theorem 3.2]{Karpilovsky}).  

Schur showed that each cohomology class contains only finitely many $\sim_\GL$--classes of \emph{irreducible} projective representations.
This is proven in three steps.  First, one observes that $\approx_\GL$--classes of projective representations with cocycle $\sigma$ are in bijection with isomorphism classes of modules over the \emph{twisted group algebra} $\bbC^\sigma [G]$, and irreducible projective representations correspond to irreducible modules.  Here $\bbC^\sigma [G]$ is the $\bbC$--algebra with basis $\{\overline{g} \,|\, g\in G\}$ and with multiplication induced by setting $\ol{g_1} \ol{g_2} = \sigma(g_1, g_2) (\ol{g_1} \cdot \ol{g_2})$  (see \cite[\S3.2]{Karpilovsky}).  

The second step is to show that for any cocycle $\sigma$, there are finitely many isomorphism classes of irreducible modules over $\bbC^\sigma [G]$.  In fact, the set of such isomorphism classes is in bijection with the set of so-called $\sigma$--regular conjugacy classes in $G$ \cite[Theorem 6.7]{Karpilovsky}.  (An extension of this result can be found in Tappe \cite{Tappe}.)
We now conclude that there are finitely many $\approx_\GL$--classes of irreducible projective representations with cocycle $\sigma$.  The final step is to check that if $\sigma'$ is cohomologous to $\sigma$, then every projective representation with cocyle $\sigma'$ is $\sim_\GL$--equivalent to a projective representation with cocycle $\sigma$ (\cite[Lemma 1.1 (ii)]{Karpilovsky}).  This shows that there are finitely many $\sim_\GL$--classes of irreducible projective representations with associated cohomology class $[\sigma]\in H^2 (G; \bbC^*)$.

\begin{remark} For ordinary representations of finite groups, the fact that there are finitely many irreducibles is often proven by observing that the group ring $\bbC[G]$ is semisimple.  
If $G$ is solvable, then a theorem of Passman~\cite[Theorem 3]{Passman} asserts that $\bbC^\sigma [G]$ is semisimple.  
In general, one could try to prove this result by the same averaging argument used to show that the ordinary group ring $\bbC[G]$ is semisimple (see, for example, Serre \cite[Chapter 6, Proposition 9]{Serre}).  However, in order to average over the group $G$, one must divide by $\sum_{g\in G} \sigma (g, g^{-1}) \in \bbC^*$ rather than by the order of $G$.  It is unclear when this sum is non-zero.
\end{remark}

Combining Corollary \ref{U-GL} and Theorem \ref{schur} yields a finiteness result for irreducible projective unitary representations of finite groups.

\begin{corollary} $\label{discrete}$ For any finite group $G$, there are finitely many projective \emph{unitary} equivalence classes ($\sim_{U}$--classes) of irreducible projective representations $\rho\co G\to U(n)$.  
Equivalently, there are finitely many irreducible elements in the moduli space $\Hom(Q, \PU(n))/\PU(n)$.
\end{corollary}

\begin{corollary} $\label{disjoint}$ The space $\rHom$ is the disjoint union (topologically) of the subspaces $\rHom \cap \Sum(\Gamma, U(n))$ and
$$\rHom_{[\psi]} \co= \{\rho\co \Gamma \to U(n) \,\, | \,\, \rho(A) \subset S^1, \,\, \overline{\rho} \isom \psi \co Q \to \PU(n)\},$$
where $\psi$ ranges over a set of representatives for the finite collection of irreducibles in $\Hom(Q, \PU(n))/\PU(n)$.  
\end{corollary}
\begin{proof}
By Lemma~\ref{irred-proj}, $\rHom$ is the disjoint union, set-theoretically, of the above spaces.  Lemma \ref{loc-cpt} tells us that $\Sum(\Gamma, U(n))$ is closed in $\Hom(\Gamma, U(n))$, so $\rHom \cap\Sum(\Gamma, U(n))$ is closed in $\rHom$.  Each subspace $\rHom_{[\psi]}$ is closed in $\rHom$ as well, because it is the inverse image of the point
$[\psi]$ under the continuous map 
$$\rHom \maps \Hom (Q, PU(n))/PU(n)$$
 sending $\rho$ to $[\brho]$.  Note here that since $PU(n)$ is compact, $\Hom (Q, PU(n))/PU(n)$ is Hausdorff (see, for example, Bredon \cite[Theorem I.3.1]{Bredon-transf}).

We have now partitioned $\rHom$ into a finite number of disjoint closed sets, and hence each must be open as well, completing the proof.
\end{proof}

We now study the subspaces $\rHom_{[\psi]}$.  Each of these subspaces is the union, over $\psi' \isom \psi$, of the subspaces
$$\rHom_{\psi'} \mathrel{\mathop :}= \{ \rho\in \rHom \, | \, \brho = \psi'\}.$$  

\begin{lemma} $\label{fibers}$ For each $\phi\co Q\to PU(n)$, the subspace 
$$\rHom_{\phi} \subset \Hom(\Gamma, U(n))$$
is a closed, semi-algebraic subset.  In particular, $\rHom_\phi$ is compact.
\end{lemma}
\begin{proof}  The fact that $\rHom_\phi$ is closed in $\Hom(\Gamma, U(n))$ follows immediately from the fact that this space is the fiber, over $\phi$, of the natural map 
$$\rHom \maps \Hom(Q, PU(n));$$ note also that $\rHom$ is closed in $\Hom(\Gamma, U(n))$ (Lemma \ref{proj-subvar}).  

To see that $\rHom_\phi$ is semi-algebraic, choose generators $\gamma_1, \ldots, \gamma_r$ for $\Gamma$ and representatives $P_i \in U(n)$ for $\phi ([\gamma_i]) \in PU(n)$.  One then obtains a description of $\rHom_\psi$ involving quantified polynomial equations, since a representation $\rho\co \Gamma\to U(n)$ will satisfy $\overline{\rho} = \psi$ if and only if it is scalar on $A$ and there exist $\lambda_i \in S^1$ with $\rho(\gamma_i) = \lambda_i P_i$ ($i=1, \ldots, r$).  Elimination of quantifiers (see Section \ref{real}) then shows that $\rHom_\psi$ is semi-algebraic.
\end{proof}

If $\psi' = P\psi P^{-1}$ for some $P\in \PU(n)$, then $P$ and $P^{-1}$ induce inverse homeomorphisms between $\rHom_\psi$ and $\rHom_{\psi'}$.  In fact, we will now show that $\rHom_{[\psi]}$ is a locally trivial fiber bundle with these fibers.  

\begin{proposition}$\label{bundle}$  Let $\psi \co Q \to \PU(n)$ be a homomorphism for which the subspace $\rHom_{[\psi]}$ is non-empty.  Then the map
$$\rHom_{[\psi]} \srm{\pi} \PU(n)\cdot \psi \subset \Hom(Q, \PU(n)),$$
given by $\pi (\rho) = \overline{\rho}$, is a fiber bundle over the orbit of $\psi$, with fiber $\rHom_{\psi}$.
\end{proposition}
\begin{proof} Let $\ol{\Stab} (\psi) \subset \PU(n)$ be the stabilizer of $\psi$ under the $\PU(n)$--action, and let $\Stab(\psi)$ denote the inverse image of $\ol{\Stab}(\psi)$ in $U(n)$.  Note that both of these subgroups are closed.  Since $U(n)$ is a compact Lie group and $\Stab (\psi)$ is a closed subgroup, the quotient map $q_\psi\co U(n) \to U(n)/\Stab (\psi)$ is a principal $\Stab (\psi)$--bundle (see, for example, Duistermaat and Kolk \cite[1.10.7 and 1.11.4]{Duis-Kolk}).  Hence $q_\psi$ admits local sections (since we only need the existence of local sections, Gleason's Theorem \cite{Gleason} could be used here).  Choose an open covering $\{V_i\}_i$ of $U(n)/\Stab (\psi)$ admitting local sections $\alpha_i \co V_i \to U(n)$.  Since $U(n)/\Stab (\psi)$ is a Lie group, it is regular (in fact, quotients of regular spaces by compact groups are always regular; see Munkres \cite[Exercise 31.8]{Munkres}), so we may choose an open cover $\{U_j\}_j$ of $U(n)/\Stab (\psi)$ such that for each $j$ there exists $i$ with $\ol{U_j} \subset V_i$.  

We have homeomorphisms 
$$U(n)/\Stab(\psi) \srt{\homeo} \PU(n)/\ol{\Stab}(\psi) \homeo  \PU(n)\cdot \psi,$$
which we treat as identifications.
So we will consider $\{U_j\}_j$ and $\{V_i\}_i$ as coverings of $\PU(n)/\ol{\Stab}(\psi)$ and $\PU(n)\cdot \psi$.

We claim that $\pi$ is trivial over each open set $U_j$.  In fact, we will show that $\pi$ is trivial over the closures of the $U_j$.  
Since $\ol{U_j} \subset V_i$ for some $i = i(j)$, if we set $\beta_j = \alpha_i$ then $\beta_j$ is a section of $q_\psi$ over $\ol{U_j}$.
We have continuous, fiber-preserving maps
$$\phi_j \co \overline{U_j} \cross \rHom_{\psi} \maps \pi^{-1} (\overline{U_j})$$
given by $\phi_j (u,  \rho) = \beta_j (u) \rho \beta_j (u)^{-1},$
which we claim are homeomorphisms.  

The spaces $\overline{U_j}$ and $\rHom_{\psi}$ are compact, since they are closed (respectively) in the compact spaces $\PU(n)/\ol{\Stab} (\psi)$ and $\rHom$ (Lemma \ref{proj-subvar}).
So the domain of $\phi_j$ is compact, and since its range is Hausdorff, it will suffice to check that $\phi_j$ is a bijection.

First we check that $\phi_j$ is surjective.  Consider a representation $\rho \in \pi^{-1} (\overline{U_j})$.  Then $\ol{\rho} = [X] \psi [X]^{-1}$ for some $X\in U(n)$ with $q_\psi (X) \in \ol{U_j}$ (where $\ol{U_j}$ is viewed as a subspace of $U(n)/\Stab(\psi)$).
Let $u = q_\psi (X)\in \ol{U_j}$.  
Now $\beta_j (u) = X K$ for some $K\in \Stab (\psi)$, and we have
$$[\beta_j (u)] \psi [\beta_j (u)]^{-1} = [X] [K] \psi [K]^{-1}  [X]^{-1} = [X] \psi [X]^{-1}  =\ol{\rho},$$ 
or in other words 
$\psi = [\beta_j (u)]^{-1} (\brho) [\beta_j (u)]$.  Now $\beta_j (u)^{-1} \rho \beta_j(u) \in \rHom_\psi$, and $\phi_j (u, \beta_j (u)^{-1} \rho \beta_j(u)) = \rho$.

Next, we check that $\phi_j$ is injective.  If $\phi_j (u, \rho) = \phi_j (u', \rho')$, then 
\begin{equation}\label{alpha}\beta_j(u')^{-1}  \beta_j(u)\rho \beta_j(u)^{-1}  \beta_j(u')= \rho'.\end{equation}
Since $\ol{\rho} = \ol{\rho'} = \psi$, we see that $\beta_j(u')^{-1}  \beta_j(u) \in \Stab (\psi)$.  But since $\beta_j$ is a section of $q_\psi$, this implies that $u' = u$, and by (\ref{alpha}), we have $\rho' = \rho$ as well, proving injectivity.

\end{proof}

We now study the individual fibers $\rHom_\psi \subset \rHom$ of the bundle from Proposition \ref{bundle}.  Each fiber admits a restriction map to $\Hom(A, S^1)$.  Note that when $A$ is a free abelian group, $\Hom(A, S^1)$ is a torus of dimension $\mathrm{rk} (A)$.  

\begin{proposition} $\label{finite-cover2}$ For each $\psi\co Q\to U(n)$, the restriction map 
$$R\co \rHom_\psi \maps \Hom(A, S^1)$$
is a (non-surjective) finite covering map with structure group $\Hom(Q, S^1)$.
\end{proposition}
\begin{proof}  The action of $\Hom(Q, S^1)$ on $\rHom_\psi$ is given by 
$$(\chi \cdot \rho)(\gamma) = \chi([\gamma])\rho(\gamma).$$
Since $\chi([\gamma])$ is central in $U(n)$, we have $\overline{\chi \cdot \rho} = \overline{\rho} = \psi$, and it also follows that $\chi\cdot \rho$ is a homomorphism with $(\chi \cdot \rho) (A) \subset S^1$.  So we have a well-defined action, which is free because $(\chi\cdot \rho) (\gamma) = \rho(\gamma)$ implies that $\chi([\gamma]) = 1$ for all $\gamma\in \Gamma$.  Hence the quotient map for this action is a covering map whose  structure group is the finite group $\Hom(Q, S^1)$.

If $a\in A$, then $(\chi\cdot\rho)(a) = \chi(1)\rho(a) = \rho(a)$, so the restriction map $R$ factors through the quotient space for this action.  We must show that the induced map
$$\overline{R}\co \left(\rHom_\psi\right)/\Hom(Q, S^1) \maps R\left(\rHom_\psi\right)$$
 is a homeomorphism.  
Recall that $\rHom_\psi$ is compact (Lemma \ref{fibers}), so the domain of $\overline{R}$ is compact as well.  Since the range of $\overline{R}$ is Hausdorff, to show that $\overline{R}$ is a homeomorphism we need only show that it is injective.

Say $\overline{R}(\rho) = \overline{R}(\rho')$.  Then we know that $\overline{\rho} = \overline{\rho'} = \psi$ and $\rho|_A = \rho'|_A$.  The first condition implies that for any $\gamma\in \Gamma$, we have $\rho(\gamma) = \lambda(\gamma) \rho'(\gamma)$, for some $\lambda(\gamma) \in S^1$, and the second condition implies that $\lambda(a) = 1$ if $a\in A$.  We simply need to check that $\lambda\co \Gamma \to S^1$ is a homomorphism.  For any $\gamma\in \Gamma$, we have $\lambda(\gamma) = \rho(\gamma) \rho'(\gamma)^{-1}$.  Now 
\begin{equation*}
\begin{split}
\lambda(\gamma_1 \gamma_2) &= \rho(\gamma_1 \gamma_2) \rho'(\gamma_1 \gamma_2)^{-1}  
								 = \rho(\gamma_1) \rho(\gamma_2) \rho'(\gamma_2)^{-1} \rho'(\gamma_1)^{-1}\\
						&     = \rho(\gamma_1) \lambda (\gamma_2)  \rho'(\gamma_1)^{-1}
								= \rho(\gamma_1)  \rho'(\gamma_1)^{-1} \lambda(\gamma_2)     = \lambda(\gamma_1) \lambda(\gamma_2).									     
\end{split}
\end{equation*}									     
\end{proof}

Next, we consider the image of the restriction map
$$R\co \Hom(\Gamma, U(n)) \to \Hom(A, U(n)).$$
Since $[\Gamma : A] < \infty$, the Schreier Index Formula implies that $A$ is finitely generated, so $\Hom(A, U(n))$ has the structure of an algebraic set, and $R$ is an algebraic map (in fact, $R$ is simply a projection).  The next result follows immediately from the Tarski--Seidenberg Theorem and Lemma \ref{fibers}.

\begin{proposition} \label{proj-var} For each homomorphism $\psi\co Q\to \PU(n)$, the image 
$$R(\rHom_\psi)\subset \Hom(A, S^1)$$
is a closed semi-algebraic subset.
\end{proposition} 


\section{The moduli space of irreducible representations}$\label{proofs}$

In this section, $\Gamma$ will denote an infinite discrete group sitting in an extension
$$1 \maps A \maps \Gamma \maps Q\maps 1$$
with $A$ a free abelian group of rank $k>0$ and $Q$ a finite group.  (Recall from Section~\ref{ind-proj} that if $\Gamma$ is virtually $\bbZ^k$, then $\Gamma$ sits in such an extension.)


\subsection{Triangulations}$\label{triang}$

We now apply the results on triangulations in Section \ref{triang-sec} to spaces of representations.

\begin{theorem} $\label{ind-triang}$ 
The moduli space $\bHom_n (\Gamma)$ admits a triangulation in which each of the following subsets is a union of open simplices:
\begin{enumerate}
\item $\bSum_n(\Gamma)$,

\item $\left(\rHom_{[\psi]}\right)/U(n)$ with $\psi$ irreducible, and

\item $\bHom_n (\Gamma)_H$ with $[\Gamma : H] < \infty$.
\end{enumerate}

\end{theorem}
\begin{proof}  It follows from Lemma \ref{loc-cpt}, Corollary \ref{Ind_H}, and Lemma \ref{proj-subvar} that $\bSum_n(\Gamma)$, $\rHom/U(n)$, and $\bHom_n (\Gamma)_H$ are all semi-algebraic subsets of $\bHom_n (\Gamma)$.  By Theorem \ref{triang-thm}, $\bHom_n (\Gamma)$ admits a triangulation in which each of these subsets is a union of open simplices.  
In fact, $\bSum_n (\Gamma)$ and $\rHom/U(n)$ are closed, so these must in fact be closed subcomplexes.
By Corollary \ref{discrete}, the closed subsets $\left(\rHom_{[\psi]}\right)/U(n) \subset \rHom/U(n)$ are topologically disjoint from one another, so each must be a closed subcomplex in its own right.
\end{proof}

\begin{proposition}$\label{k-diml}$ For any $\psi\co  Q\to \PU(n)$, the space
$\rHom_\psi$ is homeomorphic to a finite simplicial complex of dimension at most $k$.  
\end{proposition}
\begin{proof} By Proposition \ref{finite-cover2}, $\rHom_\psi$ is a finite cover of
$$R \left( \rHom_\psi \right) \subset \Hom(A, S^1)$$
which is a closed, semi-algebraic subset (Proposition \ref{proj-var}).  
By Theorem \ref{triang-thm}, there exists a triangulation of $\Hom(A, S^1)$ with $R \left( \rHom_\psi \right)$ as a subcomplex.  Since $A$ is a free abelian group of rank $k$, $\Hom(A, S^1)\homeo (S^1)^k$.   Hence $\Hom(A, S^1)$ is a $k$--dimensional manifold, and any triangulation of $\Hom(A, S^1)$ must be $k$--dimensional.  

To complete the proof, note that any covering space of a $d$--dimensional simplicial complex is again a $d$--dimensional simplicial complex (see Seifert and Threlfall \cite[\S 55]{ST}, or Spanier \cite[\S 3.8, Theorem 3]{Spanier}.)
\end{proof}


\subsection{The moduli space}$\label{moduli}$

In this section, we prove our main result regarding the moduli space of irreducible representations.
We will need the following lemma, which is a simple consequence of Brouwer's Invariance of Domain theorem.

\begin{lemma} $\label{inv}$ Let $X$ be a topological space with open subsets $U, V \subset X$.  If $U\cap V\neq \emptyset$ and there exist  homeomorphisms $f \co U \srt{\homeo} \bbR^n$, $g\co V \srt{\homeo} \bbR^m$, then $n = m$.
\end{lemma}

\begin{theorem}$\label{moduli-k-diml}$ For any $n>0$, 
the one-point compactification $\bIrrp_n (\Gamma)$ of the moduli space of irreducible $U(n)$--representations of $\Gamma$ is homeomorphic to a CW complex of dimension at most $k = \rk (A)$.  
\end{theorem}
\begin{proof}  Consider a triangulation of $\bHom_n (\Gamma)$ as in Theorem \ref{ind-triang}.  We will prove, by induction on $|Q|$, that any open simplex of dimension greater than $k$ in such a triangulation must lie in $\bSum_n (\Gamma)$.  Modding out the closed subcomplex $\bSum_n (\Gamma)$ then yields the desired CW structure on $\bIrrp_n (\Gamma)$.

If $|Q| = 1$, then $\Gamma = A\isom \bbZ^k$, so $\bIrrp_n (\Gamma)$ is a point for $n>1$ (Proposition~\ref{max-dim}).  When $n=1$, $\bHom(A, S^1)  \homeo (S^1)^k$ and 
$\bIrrp_1 (A) = \Hom(A, S^1)_+$.  

We now assume the result for all extensions
$A' \maps \Gamma' \maps Q'$
(with $A'$ free abelian of finite rank) such that $|Q'|<|Q|$.  In particular, we assume that for each subgroup $H < \Gamma$ with $A\leqs H$ and $n | [\Gamma: H]$, and each semi-algebraic triangulation of $\bHom_{n/[\Gamma : H]} (H)$ as in Theorem \ref{ind-triang}, all simplices of dimension greater than $k$ lie in the subspace of reducible representations.

Let $f\co K \srt{\homeo} \bHom_n (\Gamma)$ be a triangulation as in Theorem \ref{ind-triang}.  
We begin by proving that any simplex of dimension greater than $k$ must lie either in $\bSum_n (\Gamma)$ or in $\left( \rHom_{[\psi]} \right)/U(n)$ for some irreducible $\psi$.  
By Lemmas \ref{Serre8.1} and \ref{irred-proj}, the complement of these subcomplexes is contained in the union, over $A\leqs H<\Gamma$ with $n | [\Gamma : H]$, of the subspaces 
$$\bHom_n (\Gamma)_H = \bInd_H^\Gamma (\bIrr_{n/[\Gamma : H]} (H)).$$ 
Our induction hypothesis gives a semi-algebraic triangulation of $\bHom_{n/[\Gamma : H]} (H)$ in which the subspace of irreducibles is a union of open simplices of dimension at most $k$.  Since $\bInd_H^\Gamma$ is a semi-algebraic map (Proposition \ref{Ind-sa}) and semi-algebraic maps do not increase dimension (Corollary \ref{tame}) , it follows that the semi-algebraic open triangulation of $\bHom_n (\Gamma)_H$ induced by our triangulation $f$ is at most $k$ dimensional.

Hence any simplex of $\bHom_n (\Gamma)$ (in the triangulation $f$) with dimension greater than $k$ must lie either in $\bSum_n (\Gamma)$ or in $\left( \rHom_{[\psi]} \right)/U(n)$ for some irreducible $\psi$, and we must rule out the latter possibility.

Let $\sigma\subset \left(\rHom_{[\psi]}\right)/U(n)$ ($\psi$ irreducible) be a maximal \emph{open} simplex of this subcomplex.  We need to show that  the dimension $d$ of $\sigma$ is at most $k$.
Each point in $\sigma$ corresponds to an irreducible representation $\rho$, and Schur's Lemma tells us that the stabilizer of $\rho$ in $U(n)$ is just $S^1$.

The inverse image $\wt{\sigma}$ of $\sigma$ in $\rHom_{[\psi]} \subset \Hom(\Gamma, U(n))$ is a metric space with a free action of the compact Lie group $PU(n)$, so Gleason's slice theorem \cite{Gleason} implies that the projection $\wt{\sigma} \to \sigma$ is a principal $PU(n)$--bundle.  Thus each point $x\in \sigma$ has open subset $U \subset \sigma$ whose inverse image in $\rHom_{[\psi]}$ is an open subset of $\rHom_{[\psi]}$ homeomorphic to $U \cross PU(n)$ (note that since $\sigma$ is maximal, it is open in $\left(\rHom_{[\psi]}\right)/U(n)$).  Choosing $x$ in the interior of $\sigma$, and shrinking $U$ if necessary, we may assume that $U$ is homeomorphic to $\bbR^d$.  Since $PU(n)$ is a manifold of dimension $n^2 -1$, the inverse image of $U$ contains an open subset of $\rHom_{[\psi]}$ homeomorphic to $\bbR^{d+n^2 - 1}$.

We have shown that $\rHom_{\psi}$ is homeomorphic to a simplicial complex of dimension at most $k$ (Proposition \ref{k-diml}). Moreover, by Proposition~\ref{bundle}, $\rHom_{[\psi]}$ is a locally trivial fiber bundle over the manifold $\PU(n)/\Stab(\psi)$ with fiber $\rHom_{\psi}$.  Since $\PU(n)/\Stab(\psi)$ has dimension at most $n^2-1$, we can cover $\PU(n)/\Stab(\psi)$ by open subsets homeomorphic to $\bbR^l$ for some $l\leqs n^2-1$, over which $\rHom_{[\psi]}$ is trivial.  Each fiber $\rHom_{\psi}$ contains a dense subset (the maximal open simplices) in which each point has a neighborhood (open in the fiber) homeomorphic to some $\bbR^p$ with $p\leqs k$.  Hence 
we can find dense subset $D$ of $\rHom_{[\psi]}$ such that each point in $D$ has a neighborhood in $\rHom_{[\psi]}$ which is homeomorphic to $\bbR^{j}$ for some $j\leqs k+n^2 - 1$.  Above, we found an open subset of $\rHom_{[\psi]}$ homeomorphic to $\bbR^{d+n^2 - 1}$.  By Lemma~\ref{inv},   
$d+n^2 - 1 = j$, and since $j \leqs k+ n^2 -1$, we conclude that $d \leqs k$.
\end{proof}

\begin{remark} $\label{max-dim'l}$ The bound in Theorem~\ref{moduli-k-diml} is in fact optimal, in the following sense.  If $\Gamma$ has a normal subgroup $A\isom \bbZ^k$ of index $q$, then $\bIrrp_q (\Gamma)$ is a CW complex of dimension exactly $k$.  This can be proven as follows.  The group $\Gamma/A$ acts by conjugation on $ \Hom(A, S^1)$, and Frobenius Reciprocity implies that $\Ind_A^\Gamma (\chi)$ is irreducible if and only if $\Gamma/A$ acts freely on the orbit of $\chi$.  Moreover, the union $U$ of all such orbits is   open in $\Hom(A, S^1)\isom (S^1)^k$.  Let $\bIrr_q (\Gamma)_A \subset \bIrr_q (\Gamma)$ denote those irreducibles induced from $A$.  Applying Frobenius Reciprocity again, one may identify $\bInd\co U \surjects \bIrr_q (\Gamma)_A$ with the quotient map for the action of $\Gamma/A$, which is a covering map since $\Gamma/A$ acts freely on $U$.  (Some care is required here.  To check that the continuous bijection $f\co U/(\Gamma/A)\to \bIrr_q (\Gamma)_A$ is closed, one considers the extension $\tilde{f} \co \Hom(A, S^1)/(\Gamma/A)\to \bHom_q (\Gamma)$ and uses the facts that $\Hom(A, S^1)/(\Gamma/A)$ is compact  and $\tilde{f}^{-1} (\bIrr_q (\Gamma)_A) = U$).  It follows that $\Irr_q (\Gamma)$ contains an open set homeomorphic to $\bbR^k$.  Further details are left to the reader.
\end{remark}

 
\section{Periodicity in stable representation theory}$\label{period}$

Let $\Gamma$ be a finite extension of a free abelian group $A\isom \bbZ^k$.
We now combine our results on the moduli space $\bIrrp_n (\Gamma)$ with Lawson's work on deformation $K$--theory \cite{Lawson-prod-form, Lawson-simul} to show that $\K (\Gamma)$ is $2$--periodic above dimension $k-2$.  We begin with some definitions and results from \cite{Lawson-prod-form, Lawson-simul}.
For a discussion of the relevant background on ring and module spectra, see Ramras \cite[\S6]{Ramras-stable-moduli}.

Given an abelian topological monoid $A$, there is an associated spectrum $\Sp (A)$, constructed as follows.  Let $BA$ denote the (simplicial) classifying space of $A$.  Since $A$ is  abelian, the multiplication map $A\cross A\to A$ is a  homomorphism of monoids, and therefore yields a multiplication $BA \cross BA \homeo B (A\cross A) \to BA$ making $BA$ into an abelian topological monoid.  The spaces in the spectrum $\Sp (A)$ are then
$$\Omega BA,\,\,\, BA, \,\,\,BBA,\, \ldots,$$ 
where $\Omega$ denotes the based loop space.  Note that if $M$ is a homotopy commutative topological monoid with $\pi_0 M$ a group, then the natural map $M\to \Omega BM$ is a weak equivalence (see, for example,~\cite[\S 1]{May-EGP}).  Since $B^k A$ is path connected for $k>0$, we conclude that the spectrum $\Sp (A)$ is an $\Omega$--spectrum.

\begin{definition} $\label{Rep}$ The \emph{deformation representation ring} of $\Gamma$, denoted $\Rdef (\Gamma)$, is the spectrum $\Sp ( \Rep (\Gamma))$ associated to the topological abelian monoid
$$\Rep (\Gamma) \mathrel{\mathop :}= \coprodmo_{n=0}^\infty \bHom_n (\Gamma).$$
\end{definition}

\begin{theorem}[Lawson] $\label{Bott}$ There is a homotopy cofiber sequence of spectra
\begin{equation} \label{Bott-seq}\Susp^2 \K (\Gamma) \srm{\beta} \K (\Gamma) \maps \Rdef (\Gamma),
\end{equation}
where $\beta$ denotes the Bott map in deformation $K$--theory.
\end{theorem}

Here $\Sigma^2$ denotes the second suspension, and the Bott map $\beta$ is induced from the ordinary Bott map $\beta\co  \Susp^2 \ku \to \ku$ (in connective $K$--theory) by smashing over $\ku$ with $\K(\Gamma)$.  This uses the $\ku$--module structure on $\K(\Gamma)$ constructed in \cite{Lawson-prod-form}.
We note that there is an Atiyah--Hirzebruch style spectral sequence, which we call the Bott spectral sequence, arising from the tower of fibrations
\begin{equation} \label{Bott-tower} \cdots \srm{\Susp^4 \beta} \Susp^4 \K (\Gamma_k) \cdots \srm{\Susp^2 \beta} \Susp^2 \K (\Gamma_k) \srm{\beta} \K(\Gamma_k).
\end{equation}

We say that $\K (\Gamma)$ is \emph{periodic} in dimension $*$ if the map 
$$\beta_*\co \pi_{*} \K(\Gamma)\isom  \pi_{*+2}  \Susp^2 \K(\Gamma) \maps \pi_{*+2} \K (\Gamma)$$
is an isomorphism.  Theorem \ref{Bott} shows that periodicity of $\K (\Gamma)$ is controlled by the homotopy groups $\pi_* \Rdef (\Gamma)$.  These homotopy groups are in turn linked to the homology of the moduli space $\bIrrp_n (\Gamma)$ by the next result.

\begin{theorem}[Lawson] $\label{cofiber}$
There is a tower of fibration sequences of spectra
$$\xymatrix{ \cdots \ar[r] & \Sp (\Rep (\Gamma, n-1)) \ar[r] \ar[d]^{q_{n-1}} & \Sp  (Rep (\Gamma, n)) \ar[r] \ar[d]^{q_n} & \cdots \\
			& 		\Sp \left(\Sym^\infty( \bIrrp_{n-1} (\Gamma))\right)  		&	\Sp \left(\Sym^\infty( \bIrrp_{n} (\Gamma))\right)  
		}
$$
(i.e. the homotopy fiber of each $q_n$ is $\Sp (\Rep (\Gamma, n-1))$) and $\Rdef (\Gamma)$ is the homotopy colimit of the top horizontal sequence.
\end{theorem}

In this theorem, $\Sym^\infty (X)$ is the infinite symmetric product, which may be viewed as the free abelian topological monoid on $X$, and
 $\Rep(\Gamma, n)$ is the submonoid of $\Rep (\Gamma)$ generated by representation of dimension at most $n$.  Hence $\Rep (\Gamma, n) \subset \Rep (\Gamma)$ consists of representation whose irreducible summands have dimension at most $n$ (which makes sense by Schur's Lemma).  Note that the Generalized Dold--Thom Theorem (Dold--Thom~\cite[Section 6]{Dold-Thom} and Lima-Filho~\cite[Theorem 4.4 and Remark 4.3]{Lima-Filho}) states that $\pi_i \Sp \left(\Sym^\infty (X)\right)\isom \wt{H}_i (X)$, for any based CW complex $X$.   
The map 
$$q_n \co \Sp \left(\Rep (\Gamma, n) \right) \maps \Sp \left( \Sym^\infty (\bIrrp_n \Gamma) \right)$$
is induced by the natural map $p_n\co \Rep (\Gamma, n) \to \Sym^\infty (\bIrrp_n \Gamma)$ sending a representation $\rho$ to the unordered list of its $n$--dimensional irreducible summands.   Theorem \ref{cofiber} reflects the fact that the kernel of $p_n$ is precisely $\Rep(\Gamma, n-1)$.


\begin{corollary}$\label{hom-per}$ If $\Gamma$ is an infinite discrete group with $H_* (\bIrrp_n (\Gamma); \bbZ) = 0$ for all $n\geqs 0, *>k$, then $\pi_* \Rdef (\Gamma) = \pi_*  \Sp (\Rep(\Gamma)) = 0$ for $*>k$, and 
the Bott map
$$\beta_*\co \pi_* \K (\Gamma) \maps \pi_{*+2} \K (\Gamma)$$
is an isomorphism for $*> k-2$, and injective for $* = k-2$.
\end{corollary}  
\begin{proof} The cofiber sequence (\ref{Bott-seq}) induces a long exact sequence in homotopy  
$$\cdots \maps \pi_{*+1} \Rdef (\Gamma)\maps \pi_{*-2} \K (\Gamma) \maps \pi_* \K (\Gamma) \maps \pi_* \Rdef (\Gamma) \maps \cdots,$$
so it suffices to show that $\pi_* \Rdef (\Gamma)  = 0$ for $*>k$.   Theorem \ref{cofiber} yields
\begin{equation}\label{colim}\pi_* \Rdef (\Gamma) = \colim_n \pi_* \Sp (\Rep (\Gamma, n)),\end{equation}
together with long exact sequences
\begin{equation*}
\begin{split} \cdots \maps \pi_{*+1} \Sp \left(\Sym^\infty( \bIrrp_{n} (\Gamma))\right) \maps \pi_* \Sp &(\Rep (\Gamma, n-1)) \maps \pi_* \Sp (\Rep (\Gamma, n))\\
& \maps \pi_*\Sp \left(\Sym^\infty( \bIrrp_{n} (\Gamma))\right) \maps \cdots.
\end{split}
\end{equation*}
As noted above, $\pi_* \Sp \left(\Sym^\infty( \bIrrp_{n} (\Gamma))\right) \isom \wt{H}_* (\bIrrp_{n}(\Gamma); \bbZ)$, so these groups are zero for $*>k$.  Hence for $*>k$ and $n\geqs 0$, we have isomorphisms
\begin{equation} \label{rep-isom} \pi_* \Sp (\Rep (\Gamma, 0)) \srm{\isom} \ldots \srm{\isom} \pi_* \Sp (\Rep (\Gamma, n-1)) \srm{\isom} \pi_* \Sp (\Rep (\Gamma, n)).
\end{equation}
Since $\Rep (\Gamma, 0)$ consists of the $0$--dimensional representation only, $\Sp (\Rep(\Gamma, 0)) = *$, and $\pi_* \Sp (\Rep(\Gamma, 0)) = 0$ for all $*$.  We now conclude from (\ref{rep-isom}) that for $*>k$ and $n\geqs 0$, $\pi_* \Sp (\Rep (\Gamma, n)) = 0$.  Now (\ref{colim}) completes the proof.
\end{proof}

Our periodicity theorem now follows from Theorem~\ref{moduli-k-diml} and Corollary~\ref{hom-per}.

\begin{theorem} $\label{per-thm}$
If $\Gamma$ is virtually $\bbZ^k$, $k\geqs 0$, 
then the Bott map
$$\beta_*\co \pi_* \K (\Gamma) \maps \pi_{*+2} \K (\Gamma)$$
is an isomorphism for $*> k-2$, and injective for $* = k-2$.
\end{theorem}

By Lemma \ref{translations}, this result applies to all crystallographic groups.
Recall from the introduction that if $\Gamma$ is crystallographic and torsion-free, then $k$ is the (rational) cohomological dimension of $\Gamma$, while in general $k$ is the virtual rational cohomological dimension.  As explained in the introduction, this bound is closely analogous to the bound appearing in the Quillen--Lichtenbaum Conjectures.
(However, if $\Gamma$ is crystallographic and torsion-free but $\bbR^k/\Gamma$ is \emph{non}-orientable, then $\qcd (\Gamma) < k$.)


\section{The stable moduli space}$\label{sm-sec}$

In the author's work on surface groups \cite{Ramras-stable-moduli}, it was shown that if $\Gamma$ is the fundamental group of a product of aspherical surfaces, then there is a weak equivalence 
\begin{equation} \label{we} \Omega B \Rep (\Gamma) \srm{\heq} \bbZ \cross \colim_{n} \bHom_n (\Gamma) \homeo \bbZ \cross \Hom(\Gamma, U)/U.
\end{equation}
between the zeroth space of the spectrum $\Rdef (\Gamma) = \Sp (\Rep (\Gamma))$ and the \emph{stable moduli space} of unitary representations (note that in  \cite{Ramras-stable-moduli}, $\Rep(\Gamma)$ was written $\overline{\Rep (\Gamma)}$).  When $\Gamma$ is the fundamental group of a compact manifold $E$ (for example, if $\Gamma$ is a torsion-free crystallographic group), then this stable moduli space is homeomorphic to the stable moduli space of flat connections on principal $U(n)$--bundles over $E$, up to gauge equivalence (see Ramras~\cite[Sections 3, 6]{Ramras-surface}). 

The weak equivalence (\ref{we}) relies on the following fact: for each representation $\rho\co \Gamma \to U(n)$, there exists a representation $\psi\co \Gamma\to U(m)$ such that $\rho \oplus \psi$ lies in the connected component of the trivial representation.  In this situation, we say that $\Rep(\Gamma)$ is \emph{stably group-like}.

\begin{question} For which crystallographic groups $\Gamma$ is $\Rep (\Gamma)$ stably group-like?
\end{question}

In general, this question has a negative answer: for example, if $\Gamma$ has Kazdhan's property (T), it follows from Wang~\cite{Wang} that $\Rep (\Gamma)$ is not stably group-like.
In this section, we offer one interesting class $\Z$ of torsion-free crystallographic groups for which $\Rep (\Gamma)$ \emph{is} stably group-like.

The class $\Z$ is defined as follows.  Let $\Z_0$ be the class of (non-trivial) finitely generated free abelian groups.  We recursively define $\Z_i$ to be the class of \emph{torsion-free} crystallographic groups $\Gamma$ that sit in an extension
\begin{equation}\label{Z_i} 1 \maps \bbZ^l \maps \Gamma \maps \Gamma_{i-1} \maps 1, \end{equation}
with $\Gamma_{i-1} \in \Z_{i-1}$ and the image of $\bbZ^l$ contained in the translation subgroup of $\Gamma$ (this is well-defined by Lemma \ref{translations}).  We set $\Z = \bigcup_i \Z_i$.  Note that the space forms $\bbR^{\dim (\Gamma)}/\Gamma$ admit an interesting geometric characterization (Proposition~\ref{itb}).

\begin{theorem}$\label{stable-moduli}$
For each $\Gamma\in \Z$, the homotopy groups of $\Hom(\Gamma, U)/U$ vanish above the dimension of $\Gamma$.  Moreover, this stable moduli space is homotopy equivalent to a finite product of Eilenberg--MacLane spaces:
$$\Hom(\Gamma, U)/U \heq \prod_{i=0}^{\dim (\Gamma)} K\left(\pi_i \left(\Hom(\Gamma, U)/U\right), i\right).$$
\end{theorem}

\begin{remark} $\label{ex-rmk}$ Groups in the class $\Z$ may be constructed as follows.  Take a group $\Gamma_{i-1} \in \Z_{i-1}$ with \emph{abelian} point group $Q$, and consider a semi-direct product $\Gamma = \bbZ^l \rtimes \Gamma_{i-1}$ in which $\Gamma_{i-1}$ acts on $\bbZ^l$ via a representation $\Gamma_{i-1} \to Q\to \GL_l (\bbZ)$ (not necessarily faithful on $Q$).  This gives a (split) extension
$$\bbZ^l \maps \Gamma \maps \Gamma_{i-1}.$$
Let $A\normal \Gamma_{i-1}$ denote the translation subgroup.  Then in the semi-direct product $\Gamma$, we see that $A$ acts trivially on $\bbZ^l$.  Hence $\bbZ^l\cross A < \Gamma$ is free abelian, with quotient $Q$. Since $Q$ acts faithfully on $A$, it acts faithfully on $\bbZ^l\cross A$, and Ratcliffe \cite[Theorem 7.4.5]{Ratcliffe} implies that $\Gamma$ is crystallographic with translation subgroup $\bbZ^l\cross A$. The semi-direct product of two torsion-free groups is always torsion-free, so $\Gamma \in \Z_i$.
\end{remark}

The proof of Theorem \ref{stable-moduli} follows the same path as the argument for surface groups in Ramras~\cite[Section 6]{Ramras-surface}.
In particular, we use the following result~\cite[Proposition 6.2]{Ramras-surface}.

\begin{lemma}$\label{stably-gplike}$ Let $\Gamma$ be a finitely generated discrete group for which $\Rep (\Gamma)$ is stably group-like. 
Then the zeroth space of the spectrum $\Rdef (\Gamma)$ associated to the monoid $\Rep (\Gamma)$ is weakly equivalent to $\bbZ\cross \Hom(\Gamma, U)/U$.
\end{lemma}

\begin{proposition} $\label{Z-sg}$ For any group $\Gamma\in \Z$ and any representation $\rho\co \Gamma \to U(n)$, there exists an integer $m>0$ such that $m\rho$ lies in the connected component of the trivial representation in $\Hom(\Gamma, U(nm))$.  In particular,  $\Rep (\Gamma)$ is stably group-like. 
\end{proposition}
\begin{proof} We work by induction over the classes $\Z_i$.  Each $\Gamma\in\Z_0$ is free abelian of finite rank.  Since commuting unitary matrices are simultaneously diagonalizable, $\Hom(\bbZ^r, U(l))$ is connected (for any $r$ and $l$) so the result is immediate in this case.

Assume the result for groups in $\Z_{i-1}$, and consider some $\Gamma \in \Z_i$, sitting in
an extension 
$\bbZ^l \maps \Gamma \maps \Gamma_{i-1}$
of the form (\ref{Z_i}).  

We will use the notation $\psi \heq \psi'$ to mean that there exists a path connecting the representations $\psi$ and $\psi'$.  
We claim that for each $n$, there exists $m$ such that
\begin{equation}\label{lemma} m\Ind_A^{\Gamma} (I_n) \heq I_{mn[\Gamma : A]}. \end{equation}
Since the translation subgroup $A < \Gamma$ is normal (Lemma \ref{translations}), the representation $\Ind_A^{\Gamma} (I_n) = \bbC[\Gamma] \otimes_{\bbC[A]} \bbC^n$ is trivial on $A$, and also on $\bbZ^l\leqs A$.  
We may now view $\Ind_A^{\Gamma} (I_n)$ as a unitary representation of the quotient group $\Gamma/ \bbZ^l \isom \Gamma_{i-1}$.  
Since $\Gamma_{i-1} \in \Z_{i-1}$, by our induction hypothesis we know that there exists an integer $m>0$ such that $m\Ind_A^{\Gamma} (I_n)$ lies in the connected component of the trivial representation in $\Hom(\Gamma_{i-1}, U(mn [\Gamma : A]))$.  This yields a path in $\Hom(\Gamma, U(mn [\Gamma:A]))$, through representations trivial on $\bbZ^l < A$, from $m\Ind_A^{\Gamma} (I_n)$ to the trivial representation.

Since $U(n)$ is connected, any two isomorphic representations are connected by a path.  By (\ref{lemma}), there exists $r$ such that $r \Ind_A^{\Gamma} (1) \heq I_{r[\Gamma : A]}$ (where $1$ denotes the trivial $1$--dimensional representation).  Now for any $\rho\co \Gamma \to U(n)$, we have
\begin{equation*}
\begin{split}
 r[\Gamma : A] \rho & \isom \rho \otimes I_{r[\Gamma : A]} \heq \rho \otimes \left(r \Ind_A^{\Gamma} (1)\right)\\
& \isom \rho \otimes  \Ind_A^{\Gamma} (I_r) 
		\isom \Ind_A^{\Gamma} \left(\Res^{\Gamma}_A (\rho) \otimes I_r \right) \isom \Ind_A^{\Gamma} \left(r \Res^{\Gamma}_A (\rho)\right),
\end{split}
\end{equation*}
where we have used the Projection Formula (\ref{projection}).

Since $r \Res^{\Gamma}_A (\rho)$ is a representation of the free abelian group $A$, we know that $r \Res^{\Gamma}_A (\rho) \heq I_{rn}$.  
Let $\psi_t \co [0, 1] \to \Hom(A, U(n))$ denote a continuous path with $\psi_0 = r \Res^{\Gamma}_A (\rho)$ and $\psi_1 = I_{rn}$.  Then $\Ind_A^{\Gamma} (\psi_t)$ is a continuous path 
$$\Ind_A^{\Gamma} \left(r \Res^{\Gamma}_A (\rho)\right) \heq \Ind_A^{\Gamma} (I_{rn}).$$  
By (\ref{lemma}), there exists $s$ such that 
\begin{equation*} \label{s} s \Ind_A^{\Gamma} (I_{rn}) \heq I_{s[\Gamma : A]rn}.\end{equation*}

Combining the displayed paths  yields a path from $sr[\Gamma : A] \rho$ to  $I_{s[\Gamma : A]rn}$.
\end{proof}

\noindent {\bf Proof of Theorem~\ref{stable-moduli}.} The result follows from Corollary~\ref{hom-per} and Proposition~\ref{Z-sg}, using the argument in the proof of~\cite[Corollary 6.4]{Ramras-surface} (which was simplified a bit in~\cite[Lemma 5.7]{Ramras-stable-moduli}).  $\hfill \Box$

\vspace{.1in}
We conclude this section with a geometric description of the space forms associated to groups $\Gamma\in \Z$.  Note that Theorem~\ref{stable-moduli} may be interpreted as a statement about the stable moduli space of flat connections on bundles over these manifolds.

Let $\T_0$ denote the family of tori $\bbR^k/L$, where $L \isom \bbZ^k$ is a rank $k$ lattice with $0<k<\infty$.  We recursively define $\T_i$ to be the family consisting of all space forms $E$ which (geometrically) fiber over a space form in $\T_{i-1}$ with flat tori as fibers (geometric fibering is defined in Ratcliffe and Tschantz \cite{R-T}). We set $\T = \bigcup_i \T_i$, and we refer to space forms in $\T$ as \emph{flat iterated torus bundles}.

%
%
%
%

%
%
%
\begin{proposition} $\label{itb}$ A discrete group $\Gamma$ lies in the class $\Z$ if and only if it is the fundamental group of a flat iterated torus bundle.
\end{proposition}
\begin{proof} Using~\cite[\S 7, Lemma 5]{R-T}, one checks  inductively that if $E\in \T_i$ then $\pi_1 E \in \Z_i$.  We will show that if $\Gamma\in \Z_i$, then $\Gamma \isom \pi_1 E$ for some $E\in \T_i$.  The case $i=0$ is trivial.  Assume the statement for $i-1$ and say $\Gamma\in Z_i$.  We have an extension 
$$\bbZ^l \maps \Gamma \maps \Gamma_{i-1},$$
 as in (\ref{Z_i}).  Say $\Gamma$ acts crystallographically on $\bbR^k$.  Then each $x\in N$ acts via translation by  some $a_x\in \bbR^n$.  Setting $V = \Span (\{a_x \,| \, x\in N\})$, Ratcliffe--Tschantz \cite[Lemma 1 and Theorem 13]{R-T}, shows that
$\Gamma/N\isom \Gamma_{i-1}$ acts effectively on $\bbR^k/V$ as a discrete group of isometries, and   
it follows that $\dim (\Gamma_{i-1}) = \dim \left(\bbR^k/V\right) = k-l$.  

Now Proposition \ref{full-rk} shows that $(\bbR^k/V)/\Gamma_{i-1}$ is a \emph{compact} space form, and by~\cite[Theorem 13]{R-T}, the space form $\bbR^k/\Gamma$ geometrically fibers over 
$(\bbR^k/V)/\Gamma_{i-1}$ with fiber $V/N$, which is a flat torus by Ratcliffe~\cite[Theorem 5.3.2]{Ratcliffe}.  Since $\Gamma_{i-1} \in \Z_{i-1}$, and $(\bbR^n/V)/\Gamma_{i-1}$ is a space form with fundamental group $\Gamma_{i-1}$, we have $(\bbR^n/V)/\Gamma_{i-1} \in \T_{i-1}$ (because space forms are determined by their fundamental groups; see Wolf \cite[Theorem 3.3.1]{Wolf}). 
\end{proof}

%

\section{Examples} $\label{examples}$
 
The computations in this section provide further analogy between the stable representation theory of crystallographic groups and the Quillen--Lichtenbaum Conjecture.  Our examples  fit the setup of Section \ref{sm-sec}, hence yield computations of the homotopy groups of stable moduli spaces of representations (or flat connections).  We will relate these homotopy groups to the cohomology of the underlying group.
 
For  $k\geqs 1$, let $\Gamma_k = \bbZ^{k} \rtimes \bbZ$, with the generator $a\in \bbZ$ acting  on $\bbZ^k$  by inversion.
Then $a^2$ acts trivially, so we have $A = \langle t_1, \ldots, t_k, a^2\rangle \isom \bbZ^{k+1}$, and $[\Gamma_k : A]=2$.   Section \ref{crystal} shows that $\Gamma_k$ is crystallographic of dimension $k+1$, with $A$ as translation subgroup.  Since $\Gamma_k$ is a semi-direct product of torsion-free groups, it is torsion-free.  In fact, $\Gamma_k$ is a flat torus bundle over a torus (see Remark \ref{ex-rmk}) so the results of Section \ref{sm-sec} apply.  Note that $\Gamma_1$ is the fundamental group of the Klein bottle.

The classifying space $B\Gamma_k$ has the form $\bbR^{k+1}/\Gamma_k \isom (S^1)^{k+1}/(\Gamma_k/A)$.  To understand the action of $\Gamma_k/A \isom \bbZ/2\bbZ$ on this torus, note that $\Gamma_k$ acts crystallographically on $\bbR^{k+1}$ via the isometries $v \stackrel{t_i}{\goesto} e_i + v$, $v  \stackrel{a}{\goesto}  \frac{e_{k+1}}{2} + Tv$, where
$$T= \left[ \begin{array}{rrrr} -1 &  &  &   \\ 
							          &  \ddots & & \\
							         & & -1 &\\
							         & & & 1
	\end{array} \right]$$
(with all off-diagonal entries zero),
so the generator of $\Gamma_k/A$ acts on $(S^1)^{k+1}$ via 
$$(z_1, \ldots, z_k, \alpha) \goesto (z_1^{-1}, \ldots, z_k^{-1}, -\alpha)$$
(where both inversion and negation are taken in $\bbC$).  
To compare $\pi_* \Rdef(\Gamma_k) \otimes \bbQ$ and $H^*(\Gamma_k; \bbQ)$ we use the following result, which comes from Grothendieck's T\^ohoku paper \cite[Section 5]{Tohoku} (see also MacDonald \cite[Part II]{MacDonald}).

\begin{theorem}$\label{rational-action}$
Let $G$ be a finite group acting on a finite CW complex $X$.  Then the projection $X\to X/G$ induces an isomorphism
$$H^* (X/G; \bbQ) \srm{\isom} H^* (X; \bbQ)^G.$$
\end{theorem}
 
The action of $\Gamma/A$ on $(S^1)^{k+1}$ has degree -1 in the first $k$ coordinates and degree 1 in the last coordinate, so by Theorem~\ref{rational-action} and the K\"{u}nneth Theorem, the rank of $H^n (\Gamma_k; \bbQ)$ is $\binom{k}{n}$ for $n$ even and $\binom{k}{n-1}$ for $n$ odd.  Hence $\bbR^{k+1}/\Gamma_k$ is orientable when $k$ is even, and non-orientable with $\qcd(\Gamma_k) = k$ when $k$ is odd.

\begin{proposition}$\label{computation}$ For $k\geqs 1$ and $*\geqs 0$, $\pi_* \Rdef (\Gamma_k)$ is finitely generated and 
 $\pi_* \Rdef(\Gamma_k)\otimes \bbQ \isom H^*(\Gamma_k ; \bbQ)$.  Hence  the Bott map $\pi_* \K (\Gamma_k) \xmaps{\beta_*} \pi_{*+2} \K  (\Gamma_k)$ is an isomorphism for $*> \qcd (\Gamma_k) - 2$.
Furthermore, $\pi_0 \Rdef (\Gamma_k)   \isom \bbZ \oplus (\bbZ/2\bbZ)^{2^k -1}$.
 \end{proposition}

Theorem~\ref{Bott} gives isomorphisms $\pi_i \Rdef (\Gamma) \isom \pi_i \K (\Gamma)$, $i=0,1$, for every finitely generated group $\Gamma$. 
To calculate $\pi_* \K (\Gamma_k)$ for $*>1$, one can use the Bott spectral sequence (see \ref{Bott-tower}).  This spectral sequence may contain non-trivial differentials for $k>1$.  In light of our conjecture that $\pi_* \K (\Gamma) \isom K^* (\Gamma)$ for torsion-free crystallographic groups $\Gamma$, and the fact that the Atiyah--Hirzebruch spectral sequence in $K$--theory collapses rationally, we expect the Bott spectral sequence for $\Gamma_k$ to collapse rationally.  

In the remainder of this section, we sketch the proof of Proposition~\ref{computation}.  Full details are left to the interested reader.  

By Proposition \ref{max-dim}, $\Gamma_k$ has no irreducible representations of dimension greater than 2.
We have $\Hom(\Gamma_k, U(1)) = \bIrr_1 (\Gamma_k)\homeo \coprod_{2^k} S^1$, and
an elementary computation shows that the closure (in $\bHom_2 (\Gamma_k)$) of the space of 2-dimensional irreducible representations,  denoted   $\tIrr_2 (\Gamma_k)$, is homeomorphic to $\left(T^{k}/C_2\right)\cross S^1$, where $(z_1, \ldots, z_k, \alpha)\in T^{k}\cross S^1$ corresponds to the representation
$\rho(z_1, \ldots, z_k, \alpha)$ given by
\begin{equation*} t_i \goesto \left[ \begin{array}{rr} z_i & 0  \\ 0 & z_i^{-1} \end{array} \right],
\,\,\, a  \goesto T_a = \left[ \begin{array}{rr} 0 & \alpha \\ 1 & 0 \end{array} \right]
\end{equation*}
and $C_2 = \bbZ/2\bbZ$ acts by inversion.  For $n\geqs 0$, Theorem \ref{rational-action} now gives
\begin{equation}\label{Irr-BG} H_n (\tIrr_2 (\Gamma_k); \bbQ) \isom H^n(\tIrr_2 (\Gamma_k); \bbQ) \isom H^n(\Gamma_k; \bbQ)\end{equation}

One finds that $\pIrr_2 (\Gamma_k) := \tIrr_2 (\Gamma_k) \setminus \bIrr_2 (\Gamma_k)$ consists of those $\rho(z_1, \ldots, z_k, \alpha)$ with $z_i \in \{\pm 1\}$, so $\pIrr_2 (\Gamma_k) \homeo \Hom(\Gamma_k, U(1))$.
Modding out this subspace yields $\bIrrp_2 (\Gamma_k)$, and for $*\geqs 3$ the resulting long exact sequence in homology yields
\begin{equation} \label{int1} H_* (\bIrrp_2 (\Gamma_k); \bbZ) \isom H_* (\tIrr_2 (\Gamma_k); \bbZ).
\end{equation}

In this case, Theorem \ref{cofiber}  gives a single homotopy (co)fiber sequence of spectra
\begin{equation}\label{cof-seq} \Sp (\Rep (\Gamma_k, 1)) \maps \Rdef (\Gamma_k) \maps \Sp (\Sym^\infty  \bIrrp_2 (\Gamma_k)).
\end{equation}
Since $\Rep (\Gamma_k, 1) \homeo \Sym^\infty \bIrrp_1 (\Gamma_k)$,
we have $\pi_* \Sp(\Rep(\Gamma_k, 1))\isom H_* \left(\coprod_{2^k} S^1; \bbZ\right)$. 
The resulting long exact sequence in homotopy, along with (\ref{int1}), gives 
\begin{equation}\label{int}\pi_* \Rdef (\Gamma_k) \isom H_* (\bIrrp_2 (\Gamma_k); \bbZ)
\isom H_* (\tIrr_2 (\Gamma_k); \bbZ)
\end{equation}
for $*\geqs 3$.  
Now (\ref{Irr-BG}) yields 
$$\rk\, \pi_* \Rdef (\Gamma_k) = \rk\, H^* (\Gamma_k; \bbZ) \,\, (*\geqs 3).$$
Furthermore, letting $\partial_1$ and $\partial_2$ denote the boundary maps for the long exact sequence in homotopy associated to (\ref{cof-seq}), we see that 
\begin{equation} \label{pi_2} 
\pi_2 \Rdef (\Gamma_k) = \ker \left(\wt{H}_2 (\bIrrp_2 (\Gamma_k); \bbZ) \xmaps{\partial_2} \wt{H}_1 (\bIrrp_1 (\Gamma_k); \bbZ)\right),
\end{equation}
\begin{equation} \label{pi_0} 
\pi_0 \Rdef (\Gamma_k) = \coker \left(\wt{H}_1 (\bIrrp_2 (\Gamma_k); \bbZ) \xmaps{\partial_1} \wt{H}_0 (\bIrrp_1 (\Gamma_k); \bbZ)\right),
\end{equation}
and $\pi_1 \Rdef(\Gamma_k)$ sits in a short exact sequence
\begin{equation} \label{pi_1} 
0\maps \coker  (\partial_2)
 \maps \pi_1 \Rdef (\Gamma_k) \maps \ker (\partial_1) \maps 0.
\end{equation}

To complete the computation, we  study the boundary maps $\partial_1$ and $\partial_0$.
For any finitely generated group $\Gamma$, the boundary maps in the long exact homotopy sequence associated to the homotopy cofiber sequence
 \begin{equation}\label{hocofib-bdry}
 \Sp (\Rep_{n-1} \Gamma) \maps \Sp (\Rep_n \Gamma) \maps \Sp \left(\Sym^\infty \bIrrp_n (\Gamma)\right)
 \end{equation}
 can be realized explicitly as follows.  The inclusions $\bSum_n (\Gamma)_+\injects \Rep(\Gamma, n-1)$ and
 $\bHom_n (\Gamma)_+\injects \Rep(\Gamma, n)$ induce maps out of the associated infinite symmetric products.  Lawson~\cite[Section 2]{Lawson-simul} shows that
\begin{equation} \label{htpy-push-out}
\xymatrix{\Sp \left( \Sym^\infty \left(\bSum_n (\Gamma)_+\right) \right) \ar[r] \ar[d]^j &
						\Sp \left( \Sym^\infty \left(\bHom_n (\Gamma)_+\right) \right) \ar[d]\\
		\Sp \Rep_{n-1} (\Gamma) \ar[r] & \Sp \Rep_{n} (\Gamma)
		}
\end{equation}
is a homotopy pushout square of spectra,
meaning that the induced map between the homotopy cofibers of the rows is a weak equivalence.  The homotopy cofiber of the top row is $\Sp (\Sym^\infty \bIrrp_n (\Gamma))$ by the Generalized Dold--Thom Theorem (specifically, see Dold--Thom~\cite[Satz 5.4]{Dold-Thom} or Lima-Filho~\cite[Theorem 5.2]{Lima-Filho}).  This yields the homotopy cofiber sequence (\ref{hocofib-bdry}), and shows that the boundary map for this sequence is the composite of the homological boundary map for the cofiber sequence 
\begin{equation}\label{hom-cofiber}\bSum_n (\Gamma)_+ \maps \bHom_n (\Gamma)_+ \maps \bIrrp_n (\Gamma)\end{equation}
with the map on homotopy induced by the map $j$ in Diagram (\ref{htpy-push-out}).

Returning to $\Gamma_k$, the homological boundary maps associated to (\ref{hom-cofiber}) (for $n=2$) are the composites of the homological boundary maps for the cofiber sequence
\begin{equation}\label{hom-cofiber2}\pIrr_2 (\Gamma)_+ \stackrel{i}{\injects} \tIrr_2 (\Gamma)_+ \srm{q} \bIrrp_2 (\Gamma)\end{equation}
with the maps on homology induced by $\pIrr_2 (\Gamma)  \stackrel{l}{\injects} \bSum_2 (\Gamma)$.  

We first analyze the homological boundary maps for (\ref{hom-cofiber2}).
The long exact sequence associated to (\ref{hom-cofiber2}) shows that
$$\ker \left(H_2 (\bIrrp_2 (\Gamma_k); \bbZ)
\xmaps{\partial_2^H}  H_1 (\pIrr_2 (\Gamma_k); \bbZ) \right) \isom H_2 (\tIrr_2 (\Gamma_k); \bbZ).$$
Moreover, the image of $\partial_2^H$ is the kernel of 
\begin{equation}\label{i_*}H_1 (\pIrr_2 (\Gamma_k)_+; \bbZ ) \srm{i_*} H_1 (\tIrr_2 (\Gamma_k)_+; \bbZ).\end{equation}
One can see explicitly that under the homeomorphism $\pIrr_2 (\Gamma_k) \homeo \coprod_{2^k} S^1$, all $2^k$ inclusions $S^1 \injects \pIrr_2 (\Gamma_k) \injects \tIrr_2 (\Gamma_k)$ are homotopic to one another (up to orientation), so we conclude that the image of $\partial_2^H$ is generated by the differences between generators of $H_1 (\pIrr_2 (\Gamma_k)_+; \bbZ )$ (assuming the proper choice of orientations).

The long exact sequence associated to (\ref{hom-cofiber2}) also yields
$$\ker \left(H_1 (\bIrrp_2 (\Gamma_k)_+; \bbZ) \xmaps{\partial_1^H} \wt{H}_0 (\pIrr_2 (\Gamma_k)_+; \bbZ) \right)
\isom \Tor \left(H_1 (\bIrrp_2 (\Gamma_k)_+; \bbZ)\right)$$
because rationally, the image of (\ref{i_*}) generates $H_2 (\tIrr_2 (\Gamma)_+; \bbQ)$ (this is seen by dualizing to cohomology and using Theorem~\ref{rational-action}).
A similar analysis shows that $\Img (\partial_1^H)$ is generated by the differences between the standard generators of $H_0 (\pIrr_2 (\Gamma_k); \bbZ)$.

Next, we must consider the 
maps on homology induced by 
$$\pIrr_2 (\Gamma_k)  \stackrel{l}{\injects} \bSum_2 (\Gamma_k)\homeo \Sym^2 (\bIrr_1 (\Gamma_k)).$$  
Given $\vec{\epsilon} = (\epsilon_1, \ldots, \epsilon_k) \in \{\pm 1\}^k$, let $[\vec{\epsilon}]\in H_1 (\pIrr_2 (\Gamma_k); \bbZ)$ denote the class represented by the loop $\alpha \goesto \rho(\vec{\epsilon}, \alpha)$; note that these classes form a basis.  
Let $S^1_{\vec{\epsilon}} \subset \bIrr_1 (\Gamma_k)$ denote the image of the loop
$\alpha\goesto \chi(\vec{\epsilon}, \alpha)$, where $\chi(\vec{\epsilon}, \alpha)$ is the character $t_i \goesto \epsilon_i$, $a\goesto \alpha$.  Then $\bIrr_1 (\Gamma_k) = \coprod_{\vec{\epsilon}\in \{\pm 1\}^k} S^1_{\vec{\epsilon}}$ and 
$$\Sym^2 (\bIrr_1 (\Gamma_k)) \homeo \left(\coprod_{\vec{\epsilon}\in \{\pm 1\}^k} \Sym^2 (S^1_{\vec{\epsilon}}) \right) \coprod \left(\coprod_{\vec{\epsilon}_1 \neq \vec{\epsilon}_2} S^1_{\vec{\epsilon}_1}\cross S^1_{\vec{\epsilon}_2}\right)\textrm{\Huge{/}} (x,y)\sim(y,x).$$
Moreover, $l_* [\vec{\epsilon}]$ is   represented by the loop 
$\alpha \goesto [\chi(\vec{\epsilon},  \sqrt{\alpha}), \chi(\vec{\epsilon}, -\sqrt{\alpha})]\in \Sym^2 (S^1_{\vec{\epsilon}})$,
which generates $\pi_1 \Sym^2 (S^1_{\vec{\epsilon}})$.  Thus $H_i (\pIrr_2 (\Gamma_k); \bbZ)\xmaps{l_*} H_i (\bSum_2 (\Gamma_k); \bbZ)$ is an isomorphism onto   $H_i \left(\coprod_{\vec{\epsilon}} \Sym^2 (S^1_{\vec{\epsilon}}); \bbZ \right)$ for $i=0,1$.  

Finally, we need to compute the maps 
$$\pi_i \Sp \left(\Sym^\infty \left(\bSum_2 (\Gamma)_+\right)\right) \srm{j_*} \pi_i \Sp  \left(\Sym^\infty \left(\bIrrp_1 (\Gamma)\right)\right)$$
for $i=0, 1$.  On $\pi_0$, $j_*$ is the map of free abelian groups induced by $\pi_0 \left(\bSum_2 (\Gamma)_+\right) \to \pi_0 \Sym^\infty (\bIrrp_1 (\Gamma))$; here the basepoint components act as the identity.  In the above notation, this map sends the component $[\Sym^2 S^1_{\vec{\epsilon}}]$ to $2[S^1_{\vec{\epsilon}}]$, where $2[S^1_{\vec{\epsilon}}]$ is interpreted as an element in the free abelian monoid on the components of $\bIrrp_1 (\Gamma)$ (this monoid is naturally isomorphic to $\pi_0 \Sym^\infty (\bIrrp_1 (\Gamma))$).  The other components in the domain of $j_*$ can be ignored, since they are not in the range of the homological boundary map.  On $\pi_1$, we can identify $j_*$ with the induced map 
$$\pi_1  \left(\Sym^\infty \left(\bSum_2 (\Gamma)_+\right)\right) \srm{j_*} \pi_1 \left(\Sym^\infty \left(\bIrrp_1 (\Gamma)\right)\right).$$
Using the fact that $\Sym^\infty (X\vee Y) \isom (\Sym^\infty X)\cross (\Sym^\infty Y)$, we have
$$\Sym^\infty \left(\bSum_2 (\Gamma)_+\right) \isom \prod_{\vec{\epsilon}}   \Sym^\infty (\Sym^2 S^1_{\vec{\epsilon}})_+ \cross   \Sym^\infty \left(\left(\coprod_{\vec{\epsilon}_1\neq \vec{\epsilon}_2} S^1_{\vec{\epsilon}_1} \cross S^1_{\vec{\epsilon}_2} \right)/\sim\right)_+$$
and only $\prod_{\vec{\epsilon}} \pi_1 \Sym^\infty (\Sym^2 S^1_{\vec{\epsilon}})_+$ is in the image of the homological boundary map.  On each factor in this product, $j_*$ is induced by the inclusion of the relevant component of $\Sym^2 (\bIrrp_1 (\Gamma_k))$ into $\Sym^\infty (\bIrrp_1 (\Gamma_k))$.  Hence on the factor $\Sym^\infty (\Sym^2 (S^1_{\vec{\epsilon}})_+)$, $j_*$ is the  the map
$\pi_1 \Sym^\infty (\Sym^2 S^1)_+ \to \pi_1 \Sym^\infty (S^1_+)$.  This map is an isomorphism, as follows from the fact that $S^1\injects \Sym^2 (S^1)$ is a homotopy equivalence, together with the fact that $\pi_1 X \to \pi_1 \Sym^\infty X$ is the Hurewicz map.  We now see that the restriction of $j_*$ to $\prod_{\vec{\epsilon}}  \pi_1 \Sym^\infty (\Sym^2 S^1_{\vec{\epsilon}})_+$ is an isomorphism onto the subgroup $\prod_{\vec{\epsilon}}  \pi_1 \Sym^\infty (S^1_{\vec{\epsilon}})_+$ inside $\pi_1 \Sym^\infty (\bIrrp_1 (\Gamma_k))$.

Putting together our computations, we see that the image of $\partial_1$ is generated by all elements of the 
form $2[S^1_{\vec{\epsilon}_i}] - 2 [S^1_{\vec{\epsilon}_j}]$, $\vec{\epsilon}_i\neq \vec{\epsilon}_j$, so (\ref{pi_0}) yields
$$\pi_0 \Rdef (\Gamma_k) \isom \coker (\partial_1) \isom \bbZ\oplus (\bbZ/2\bbZ)^{2^k - 1}.$$
Next, $\ker (\partial_2) \isom H_2 (\tIrr_2 (\Gamma_k); \bbZ)$, so (\ref{pi_2}) and (\ref{Irr-BG}) yield
$$\pi_2 \Rdef (\Gamma_k) \isom H_2 (\tIrr_2 (\Gamma_k); \bbZ),\,\,\,\,\,\, \pi_2 \Rdef (\Gamma_k)\otimes \bbQ \isom H^2 (\Gamma_k; \bbQ).$$
Finally, our computations show that $\ker (\partial_1) = \Tor (H_1 (\bIrrp_2 (\Gamma_k); \bbZ))$, and also that
$\Img (\partial_2)\subset H_1 (\bIrrp_1 (\Gamma); \bbZ) = H_1 \left(\coprod_{\vec{\epsilon}} S^1_{\vec{\epsilon}}\right)$ is generated by all elements of the form $[S^1_{\vec{\epsilon}_i}] - [S^1_{\vec{\epsilon}_j}]$, $\vec{\epsilon}_i\neq \vec{\epsilon}_j$.  Hence $\coker (\partial_2) = \bbZ$, and  (\ref{pi_1}) yields an exact sequence
$$0\maps  \bbZ \maps \pi_1 \Rdef (\Gamma_k) \maps  \Tor (H_1 (\bIrrp_2 (\Gamma_k); \bbZ)) \maps 0,$$
so $\pi_1 \Rdef (\Gamma_k)\otimes \bbQ \isom \bbQ \isom H^1 (\Gamma_k; \bbQ)$, completing the proof of Proposition~\ref{computation}.

\end{document}